\documentclass[11pt]{amsart}
\usepackage{amsmath}
\usepackage{amssymb}
\usepackage{mathrsfs}
\usepackage{dsfont}
\numberwithin{equation}{section}       

\newtheorem{theorem}{Theorem}

\newtheorem{lemma}[theorem]{Lemma}
\newtheorem{proposition}[theorem]{Proposition}

\newtheorem{theoalph}{Theorem}

\theoremstyle{definition}
\newtheorem{definition}[theorem]{Definition}

\newcommand{\RR}{\mathbb{R}}
\newcommand{\NN}{\mathbb{N}}
\newcommand{\ZZ}{\mathbb{Z}}
\newcommand{\QQ}{\mathbb{Q}}

\newcommand{\parameter}{\lambda}
\newcommand{\expansion}[1]{\langle #1 \rangle}
\newcommand{\pcs}{post\nobreakdash-critical set}
\DeclareMathOperator{\interior}{Int}

\renewcommand{\=}{ : = }
\newcommand{\partn}[1]{{\smallskip \noindent \textbf{#1.}}}

\begin{document}
\title[Topological orbit equivalence classes and logistic maps]{Topological orbit equivalence classes and numeration scales of logistic maps}
\author{Mar\'{\i}a Isabel Cortez}
\address{Mar\'{\i}a  Isabel Cortez, Departamento de Matem\'atica y Ciencia de la Computaci\'on, Universidad de Santiago de Chile, Av. Libertador Bernardo O'Higgins 3363, Santiago, Chile.}
\email{maria.cortez@usach.cl}
\author{Juan Rivera-Letelier}
\address{Juan Rivera-Letelier, Facultad de Matem\'aticas, Pontificia Universidad Cat\'olica de Chile,
Avenida Vicu\~na Mackenna 4860, Santiago, Chile}
\email{riveraletelier@mat.puc.cl }
\thanks{M. I. Cortez acknowledges financial support from proyecto Fondecyt
1100318.}
\thanks{J. Rivera-Letelier acknowledges financial support from proyecto Fondecyt
1100922.}
\keywords{orbit equivalence, numeration scale, dimension group, generalized odometer, orbit equivalence, unimodal map}

\maketitle{}
\begin{abstract}
We show that every uniquely ergodic minimal Cantor system is topologically orbit equivalent to the natural extension of a numeration scale associated to a logistic map.
\end{abstract}

\section{Introduction}
In this paper we are mainly interested on orbit equivalence classes of dynamical systems given by a minimal homeomorphism acting on a Cantor set.
Such a system is called \emph{minimal Cantor system}.
Two minimal Cantor systems are (\emph{topologically}) \emph{orbit equivalent} if there exists an orbit-preserving homeomorphism between their phase spaces.
An orbit-preserving homeomorphism
induces an affine homeomorphism between the corresponding spaces of invariant probability measures endowed with the weak* topology.
However, there are minimal Cantor systems which are not orbit equivalent and yet their corresponding spaces of invariant probability measures are affine homeomorphic.

Giordano, Putnam and Skau gave several characterizations of orbit equivalence classes of minimal Cantor systems, see \cite[Theorem 2.2]{GioPutSka95}.
Among other results, they showed that two minimal Cantor systems are orbit equivalent if and only if there is a homeomorphism between their phase spaces that induces an affine homeomorphism between the corresponding spaces of invariant probability measures.



It is thus natural to look for an explicit family of minimal Cantor systems having a representative element from each orbit equivalence class.
We call a family of minimal Cantor systems with this property \emph{full}.
Since the space of invariant probability measures of a minimal Cantor system is determined by its orbit equivalence class, up to an affine homeomorphism, a full family must realize every metrizable Choquet simplex up to an affine homeomorphism, as the space of invariant probability measures of one of its elements.

Downarowicz showed in~\cite{Dow91} that the family of $0$-$1$ Toeplitz flows realizes every metrizable Choquet simplex up to an affine homeomorphism.
Nevertheless, this family is not full, even within uniquely ergodic systems.
For example, no Toeplitz flow can be orbit equivalent to a Sturmian subshift, see~\S\ref{ss:notes}.

In \cite{CorRiv10b} we showed that the family of generalized odometers introduced by Bruin, Keller and St. Pierre in~\cite{BruKelStP97} to study post-critical sets of unimodal maps, realizes every metrizable Choquet simplex up to an affine homeomorphism.
See~\S\ref{ss:generalized odometers} for the definition of the generalized odometer associated to a unimodal map, \cite{Bru03,CorRiv10a} for other results about these systems and \cite{BarDowIwaLia00,BarDowLia02,GraLiaTic95} and references therein for more information on generalized odometers, which are also known as ``numeration scales'' or ``generalized adding machines''.

Thus, the question whether the family of (the natural extensions of) generalized odometers associated to unimodal maps is full, arises naturally.
In this paper we show that this family is full within uniquely ergodic minimal Cantor systems.
Note that a minimal Cantor system which is orbit-equivalent to a uniquely ergodic one is also uniquely ergodic.
The orbit equivalence class of such a system will be called \emph{uniquely ergodic}.
\begin{theoalph}
\label{t:fullness}
Every uniquely ergodic topological orbit equivalence class contains the natural extension of a generalized odometer associated to a unimodal map.
\end{theoalph}
The family formed by the usual odometers and by the Denjoy systems is also full within uniquely ergodic minimal Cantor systems, see~\cite[Corollary~2]{GioPutSka95} and \cite{PutSchSka86}.
However, this family is not full because it only contains uniquely ergodic systems, in contrast with the family of generalized odometers associated to unimodal maps which realizes every metrizable Choquet simplex up to an affine homeomorphism.

The notion of orbit equivalence applies without change to more general group actions.
Recently, Giordano, Matui, Putnam and Skau showed that for each integer~$d \ge 2$, every minimal continuous $\ZZ^d$-action on a Cantor set is orbit equivalent to a minimal Cantor system, see~\cite{GioMatPutSka08} and~\cite{GioMatPutSka10}.
As a corollary of this result and of Theorem~\ref{t:fullness}, we get that every uniquely ergodic minimal continuous $\ZZ^d$-action on a Cantor set is orbit equivalent to the natural extension of a generalized odometer associated to a unimodal map.

We now state a version of Theorem~\ref{t:fullness} in terms of the \textit{logistic family} of interval maps~$(f_\parameter)_{\parameter \in (0, 4]}$, defined for a parameter $\parameter \in (0, 4]$ by
\begin{center}
\begin{tabular}{rcl}
$f_\parameter : [0, 1]$ & $\to$ & $[0, 1]$ \\
$x$ & $\mapsto$ & $\parameter x (1 - x).$
\end{tabular}
\end{center}
For each $\parameter \in (0, 4]$ the derivative of~$f_{\parameter}$ vanishes precisely at $x = 1/2$.
We call $x = 1/2$ the \textit{critical point} of~$f_\parameter$ and its $\omega$\nobreakdash-limit set is called the \textit{\pcs{}} of~$f_{\parameter}$.
It is a compact set that is forward invariant by~$f_{\parameter}$.
\begin{theoalph}
\label{t:semi-equivalence}
Let~$X$ be a Cantor set and let~$T : X \to X$ be a minimal and uniquely ergodic  homeomorphism.
Then there is~$\parameter \in (0, 4]$ such that the \pcs{} $X_\parameter$ of the logistic map~$f_\parameter$ is a Cantor set, the restriction of~$f_\parameter$ to~$X_\parameter$ is minimal and such that there is a continuous orbit preserving map~$h : X \to X_\parameter$ whose inverse is defined outside the backward orbit of the critical point~$x = 1/2$ of~$f_\parameter$.
\end{theoalph}
Note that the map~$h$ in the previous theorem is automatically onto because~$f_\parameter$ is minimal on~$X_\parameter$.
However, $h$ is not likely to be a homeomorphism, because logistic maps are usually not injective on their \pcs{s}.

One of the main ingredients in the proofs of Theorem~\ref{t:fullness} and of Theorem~\ref{t:semi-equivalence} is the unital ordered group associated to a minimal Cantor system introduced by Herman, Putnam and Skau in \cite{HerPutSka92}.
Giordano, Putnam and Skau showed in~\cite[Theorem 2.2]{GioPutSka95} that this unital ordered group, modulo its infinitesimal subgroup, determines the orbit equivalence class of the corresponding minimal Cantor system.
Thus, to show that a family of minimal Cantor systems is full within uniquely ergodic systems, it is enough to show that this family realizes every acyclic countable additive subgroup of~$\RR$, up to infinitesimals, see~\S\ref{ss:complete invariant} and Lemma~\ref{l:unique state} in~\S\ref{ss:dimension groups}.
To do this we write each such group as a direct limit of non-negative matrices, which are essentially given by the iteration of a multidimensional Euclidean algorithm.\footnote{It turns out this algorithm is somewhat similar to the (homogeneous) Jacobi-Perron algorithm, see for example~\cite{Ber71,Sch73}.}

We obtain as a consequence a canonical representation of a given finitely generated subgroup of~$\RR$ as the direct limit of unimodular matrices, see Proposition~\ref{p:finitely generated} in~\S\ref{p:finitely generated} and compare with the original result of Riedel in~\cite{Rie81}.
We complete the proof of Theorem~\ref{t:fullness} and Theorem~\ref{t:semi-equivalence} by showing that the matrices appearing in the direct limit can be represented as the transition matrices of a certain Bratteli-Vershik system which is conjugated to a generalized odometer associated to a unimodal map.
This Bratteli-Vershik system was introduced by Bruin in~\cite{Bru03}.

\subsection{Notes and references}
\label{ss:notes}
See~\cite{GlaWei95, Put10} for other approaches to the results of Giordano, Putnam and Skau in~\cite{GioPutSka95}, on topological orbit equivalence for minimal Cantor systems.

Since the maximal equicontinuous factor of a Toeplitz flow is an odometer, see for example~\cite{Dow05}, it follows that the unital ordered group associated to a Toeplitz flow contains as a subgroup an acyclic subgroup of~$\QQ$, see for example~\cite[Proposition~3.1]{GlaWei95}.
So by~\cite[Theorem~2.2]{GioPutSka95} it follows that the family of Toeplitz flows is not full within uniquely ergodic minimal Cantor systems.
For example, no Toeplitz flow can be orbit equivalent to a Sturmian subshift, see for example~\cite[Proposition~3.4]{DarDurMaa00}.
See also~\cite[\S4.1]{GjeJoh00} for further results on the unital ordered group associated to a Toeplitz flow.

The generalized odometers we construct in Theorem~\ref{t:fullness} are such that their associated unital ordered group has a trivial infinitesimal subgroup.
So, if the minimal Cantor system~$(X, T)$ is such that its unital ordered group has a trivial infinitesimal subgroup, then the natural extension of the generalized odometer is strong orbit equivalent to~$(X, T)$, see~\cite[Theorem~2.1]{GioPutSka95}.

It is not clear to us if in Theorem~\ref{t:semi-equivalence} is possible to choose~$\parameter \in (0, 4]$ in such a way that the natural extension of~$(X_\parameter, f_\parameter)$ is orbit equivalent to~$(X, T)$. 

In \cite{Shu05}, Shultz introduces a dimension group associated to a piecewise monotone map~$f$ acting on the unit interval~$[0, 1]$.
This construction involves the dynamics of~$f$ on the whole interval~$[0,1]$.
In contrast, we only consider the dynamics of a special class of unimodal maps acting on~$[0, 1]$, but restricted to a strictly smaller invariant set.
It is not clear to us if there is a direct relation between the corresponding dimension groups. 

\subsection{Organization}
\label{ss:organization}
In \S\ref{s:preliminaries} we introduce some notations and recall some basic facts, including the concept of dimension group (\S\ref{ss:dimension groups}) and its relation with (non-invertible) minimal Cantor systems (\S\S\ref{ss:complete invariant}, \ref{ss:dynamical ordered group}).

In \S\ref{s:algorithm} we show that every finitely generated subgroup~$\Gamma$ of~$\RR$ of rank at least~2, endowed with the order structure induced by the usual order structure of~$\RR$, is isomorphic to a direct limit of a certain class of unimodular matrices that we call ``admissible'' (Proposition~\ref{p:finitely generated} in~\S\ref{ss:finitely generated}).
These matrices are defined by the iteration of a multidimensional Euclidean algorithm defined in~\S\ref{ss:algorithm}, having as an input a positive real vector whose coordinates form a base~$\Gamma$.

In \S\ref{s:infinitely generated} we show how to use the algorithm defined in~\S\ref{ss:algorithm} to represent a countable additive subgroup of~$\RR$ that contains~$1$, but is not contained in~$\QQ$, as a direct limit of a certain class of matrices we call ``basic'', see Proposition~\ref{p:infinitely generated}.
These matrices appear naturally as transition matrices of the Bratteli-Vershik system associated to a generalized odometer associated to a unimodal map, see~\S\ref{ss:generalized odometers} for the definition of these objects.
We deduce Theorem~\ref{t:fullness} and Theorem~\ref{t:semi-equivalence} from Proposition~\ref{p:infinitely generated} and known facts in~\S\ref{ss:logistic}.

\subsection{Acknowledgements}
We would like to thank the referee for reading carefully the paper and making suggestions that helped improve the exposition.
\section{Definitions and Background}
\label{s:preliminaries}
The purpose of this section is to fix some notations and terminology and to recall some basic results that will be used in the rest of the paper.

We denote by~$\NN = \{ 1, 2, \ldots \}$ the set of strictly positive integers.
We use the interval notation to denote subsets of~$\ZZ$, so for $n, m \in \ZZ$
$$ [n, m]
=
\begin{cases}
\{ n, n + 1, \ldots, m \} & \text{if } n \le m;
\\
\emptyset & \text{if } n \ge m + 1.
\end{cases} $$
For $i \in \NN \cup \{ 0 \}$ and $J \subset \NN$ we put~$i + J = \{ i + j : j \in J \}$.

Given $x \in \RR$ we denote by~$[x]$ the integer part of~$x$ and by $\{ x \} \= x - [x] \in [0, 1)$ its fractional part.

\subsection{Linear algebra}
\label{ss:linear-algebra}
Given a finite set~$V$ and a vector~$\vec{x} \in \RR^V$, for each~$v \in V$ we denote by~$x_v$ the corresponding coordinate of~$\vec{x}$. 
On the other hand we denote by~$\vec{e}_v$ the vector in~$\RR^V$ having all of its coordinates equal to~$0$, except for the coordinate corresponding to~$v$ which is equal to~$1$.

Unless otherwise stated, all the vectors we consider are column vectors.

Given a finite subset~$V$ of~$\ZZ$ we say that a vector~$(x_v)_{v \in V} \in \RR^V$ is \emph{non-increasing} (resp. \emph{strictly decreasing}) if for every $v, v' \in V$ such that $v' \ge v$ we have $x_{v'} \le x_v$ (resp. $x_{v'} < x_v$).

Given finite sets~$I$ and~$I'$, a \emph{$I \times I'$ matrix} means a real matrix indexed by~$I \times I'$.
For~$(i, i') \in I \times I'$ and a $I \times I'$ matrix~$M$ we denote by~$M(i, i')$ the corresponding coefficient of~$M$ and by~$M(\cdot, i')$ the corresponding column of~$M$.
We say that such a matrix~$M$ is \textit{strictly positive} (resp. \textit{non-negative}, \textit{integer}) if each of its coefficients has the same property.

Recall that a square matrix with integer coefficients is \emph{unimodular} if its inverse is defined and has integer coefficients.
\subsection{Additive subgroups of~$\RR$}
\label{ss:real-groups}

Let~$\Gamma$ be a finitely generated additive subgroup of~$\RR$.
Then~$\Gamma$ is free and therefore the elements of each base of~$\Gamma$ are rationally independent. 
On the other hand, if~$\Gamma'$ is a subgroup of~~$\Gamma$, then~$\Gamma'$ is finitely generated, free and of rank less than or equal to that of~$\Gamma$.

\subsection{Ordered groups}
\label{ss:ordered-groups}
Let~$G$ be an Abelian group written additively.
A \emph{positive cone} of~$G$ is a subset~$G^+$ verifying,
$$ (G^+) + (G^+) \subseteq G^+,
(G^+) + (-G^+) = G
\text{ and }
(G^+) \cap (-G^+) = \{ 0 \}. $$
An \emph{ordered group} is a pair $(G,G^+)$ such that~$G$ is a group and~$G^+$ is a positive cone of~$G$.
A positive cone~$G^+$ of a group~$G$ defines a partial order~$\le$ on~$G$ defined for~$a, b \in G$ by $a\le b$ if $b-a\in G^+$.

The set
$$ \RR^+ \= \{ x \ge 0 : x \in \RR \}. $$
is a positive cone of the additive group~$\RR$ which induces the usual order on~$\RR$.

Given ordered groups $(G,G^+)$ and $(H,H^+)$ we say that a group homomorphism $\phi:G\to H$ is \emph{positive} if $\phi(G^+)\subseteq H^+$.
An \emph{isomorphism of ordered groups} between $(G,G^+)$ and $(H,H^+)$ is a group isomorphism $\phi : G \to H$ such that both, $\phi$ and~$\phi^{-1}$ are positive.

An \emph{order  unit} of an ordered group~$(G,G^+)$ is an element $u\in G^+$ such that for every $g\in G$ there exists $n\in \NN$ such that $g\leq n u$.
A \emph{unital ordered group} is a triple $(G,G^+,u)$ such that $(G,G^+)$
is an ordered group and $u$ is an order unit.
A homomorphism between two unital ordered groups $(G,G^+,u)$ and $(H,H^+,v)$ is
a positive homomorphism $\phi : G \to H$ such that $\phi(u) = v$.

A \emph{state} of an ordered group~$(G, G^+)$ is a non-zero positive homomorphism from $(G, G^+)$ to $(\RR, \RR^+)$.
The \emph{infinitesimal subgroup} $\inf{G}$ of~$(G, G^+)$ is
$$ \inf(G) \= \{ g \in G : \phi(g) = 0 \text{ for every state~$g$ of~$G$} \}. $$
The quotient~$G / \inf(G)$ has a natural structure of ordered group given by the positive cone
$$ \left( G / \inf(G) \right)^+
\=
\{ [a] : a \in G^+ \}. $$
If~$u$ is an order unit of~$G$, then the projection of~$u$ in~$G/\inf(G)$ is an order unit of~$(G/\inf(G), (G/\inf(G))^+)$.

Given $m \in \NN$, the set
$$ (\ZZ^m)^+
\=
\left\{ (v_1,\cdots,v_m) \in \ZZ^m : \text{ for all } i \in \{ 1, \ldots, m \}, v_i\geq 0 \right\}. $$
is a positive cone of~$\ZZ^m$.
Note that for~$m, m' \in \NN$ a $m \times m'$ matrix~$A$ with integer coefficients induces a positive homomorphism from~$(\ZZ^{m'}, (\ZZ^{m'})^+)$ to $(\ZZ^{m}, (\ZZ^{m})^+)$ if and only if all of its coefficients are non-negative.
\subsection{Direct limits of ordered groups}
\label{ss:direct limits}
Let~$(G_n)_{n = 1}^\infty$ be a sequence of groups and for each~$n \in \NN$ let~$\phi_n : G_n \to G_{n + 1}$ be a group homomorphism.
Given~$n \in \NN$ and $v \in G_n$ we denote by~$(v, n)$ the corresponding element of the disjoint union~$\bigsqcup_{n = 1}^\infty G_n$.
Let~$\sim$ be the equivalence relation defined on $\bigsqcup_{n = 1}^\infty G_n$ by~$(v, n) \sim (v', n')$ if there is $m \ge \max \{ n, n' \}$ such that
$$ \phi_{m-1} \circ \cdots \circ \phi_{n} v
=
\phi_{m - 1} \circ \cdots \circ \phi_{n'} v'. $$
Then the quotient of~$\bigsqcup_{n = 1}^\infty G_n$ by~$\sim$ has a natural group structure and the resulting group is a direct limit of~$\left( G_n, \phi_n \right)_{n = 1}^\infty$ in the category of groups.
We denote this group by~$\varinjlim \left( G_n, \phi_n \right)_{n = 1}^\infty$ and for~$(v, n) \in \bigsqcup_{n = 1}^\infty G_n$ we denote by~$[v, n]$ its equivalence class for~$\sim$.

Let~$\left( (G_n, G_n^+) \right)_{n = 1}^\infty$ be a sequence of ordered groups and for each~$n \in \NN$ let~$\phi_n : G_n \to G_{n + 1}$ be a positive homomorphism.
Then the group~$H \= \varinjlim ( G_n, \phi_n)_{n = 1}^\infty$ together with
$$ H^+ \= \{ [v, n] : n \in \NN, v \in G_n^+ \} $$
forms an ordered group which is a direct limit of $\left( (G_n, G_n^+), \phi_n \right)_{n = 1}^\infty$ in the category of ordered groups.
By abuse of language we call $(H, H^+)$ \emph{the direct limit of~$\left( (G_n, G_n^+), \phi_n \right)_{n = 1}^\infty$} and we denoted it by
$$ \varinjlim \left( (G_n, G_n^+), \phi_n \right)_{n = 1}^\infty. $$

\subsection{Dimension groups}
\label{ss:dimension groups}
For references to this section see \cite{Eff81,Goo86}.
We say that an ordered group~$(G, G^+)$ is \emph{unperforated} if for~$g \in G$ and $n \in \NN$ the property~$n g \in G^+$ implies~$g \in G^+$.
On the other hand, we say~$(G, G^+)$ has the \emph{Riesz interpolation property} if for~$a, a', b, b' \in G$ satisfying
$$ a \le b, a \le b', a' \le b \text{ and } a' \le b' $$
there is $c \in G$ such that~$a \le c \le b$ and~$a' \le c \le b'$.
A countable ordered group $(G, G^+)$ is a \emph{dimension group} if it is unperforated and has the Riesz interpolation property.




The following lemma characterizes dimension groups with a unique state up to a scalar factor, up to infinitesimals.
It is a direct consequence of the definitions.
\begin{lemma}
\label{l:unique state}
If~$\Gamma$ is a countable additive subgroup of~$\RR$, then~$\Gamma^+ \= \Gamma \cap \RR^+$ is a positive cone of~$\Gamma$ and $(\Gamma, \Gamma^+)$ is a dimension group.
Furthermore, the inclusion of~$\Gamma$ in~$\RR$ is the unique state of~$(\Gamma, \Gamma^+)$ up to a scalar factor and the infinitesimal group of~$(\Gamma, \Gamma^+)$ is trivial.

Conversely, if~$(G, G^+)$ is a dimension group having a unique state~$\phi$ up to a scalar factor, then~$\phi(G)$ is a countable subgroup of the additive group~$\RR$ and~$\phi$ induces a positive homomorphism between $\left( G/\inf(G), (G/\inf(G))^+ \right)$ and~$(\phi(G), \phi(G) \cap \RR^+)$.
\end{lemma}
We will use the following lemma to represent dimension groups with a unique state up to a scalar factor, modulo infinitesimals.

Given a sequence~$(d_n)_{n = 1}^\infty$ in~$\NN$ and for each~$n \in \NN$ a non-negative $d_n \times d_{n + 1}$ matrix~$A_n$, we define the inverse limit
$$ \varprojlim \left( (\RR^+)^{d_n}, A_n \right)_{n = 1}^\infty, $$
as the subset of~$\prod_{n = 1}^\infty (\RR^+)^{d_n}$ of all those~$( \vec{x}^{(n)} )_{n = 1}^\infty$ such that for each~$n \in \NN$ we have $\vec{x}^{(n)} = L_n \vec{x}^{(n + 1)}$. 
\begin{lemma}
\label{l:nexo}
Let~$(d_n)_{n = 1}^\infty$ be a sequence in~$\NN$ and for each~$n \in \NN$ let~$L_n$ be a $d_{n + 1} \times d_{n}$ matrix with non-negative integer coefficients and consider the ordered group
$$ (G, G^+)
\=
\varinjlim \left( (\ZZ^{d_n}, (\ZZ^{d_n})^+), L_n \right)_{n = 1}^\infty. $$
Assume there is~$(\vec{x}^{(n)})_{n = 1}^\infty \in \prod_{n = 1}^\infty (\RR^+)^{d_n}$ such that
$$ \varprojlim \left( (\RR^+)^{d_n}, L_n^T \right)_{n = 1}^\infty
=
\left\{ \left( \lambda  \vec{x}^{(n)} \right)_{n = 1}^\infty : \lambda \ge 0 \right\} $$
and let~$\Gamma$ be the additive subgroup of~$\RR$ generated by the coordinates of each of the vectors in $(\vec{x}^{(n)})_{n = 1}^\infty$.
Then~$(G, G^+)$ is a dimension group and the function~$\tilde{\phi} : \bigsqcup_{n = 1}^{\infty} \ZZ^{d_n} \to \RR$ defined by~$\tilde{\phi}((\vec{v}, n)) = \langle \vec{v}, \vec{x}^{(n)} \rangle$ induces a function
$$ \phi : (G, G^+) \to (\RR, \RR^+) $$
which is the unique state of~$(G, G^+)$ up to a scalar factor.
In particular, $\phi$ induces an isomorphism of ordered groups between $(G/\inf(G), (G/\inf(G))^+)$ and $(\Gamma, \Gamma \cap \RR^+)$.
\end{lemma}
\begin{proof}
That~$(G, G^+)$ is a dimension group is shown in~\cite[Theorem~3.1]{Eff81}.

To show that~$\tilde{\phi}$ induces a function defined on~$(G, G^+)$, observe that if~$[\vec{v}, n]$ and~$[\vec{v}', n']$ are equivalent elements in~$\bigsqcup_{n = 1}^\infty \ZZ^{d_n}$ with~$n' \ge n$, then
$$ \langle \vec{v}', \vec{x}^{(n')} \rangle
=
\langle L_{n' - 1} \cdots L_n \vec{v}, \vec{x}^{(n')} \rangle
=
\langle \vec{v}, L_n^T \cdots L_{n' - 1}^T \vec{x}^{(n')} \rangle
=
\langle \vec{v}, \vec{x}^{(n)} \rangle. $$
This shows~$\tilde{\phi}([\vec{v}, \vec{x}^{(n)}]) = \tilde{\phi}([\vec{v}', \vec{x}^{(n')}])$, as wanted.

It remains to show that~$\phi$ is the unique state of~$(G, G^+)$ up to a scalar factor, because the rest of the assertions follow from Lemma~\ref{l:unique state}.
Let~$\psi$ be a state of~$(G, G^+)$ and for each~$n \in \NN$ put
$$ \vec{w}^{(n)}
\=
\left( \psi([\vec{e}_1, n]), \ldots, \psi([\vec{e}_{d_n}, n]) \right)
\in
(\RR^+)^{d_n}, $$
so that for each~$[\vec{v}, n] \in G$ we have $\psi([\vec{v}, n]) = \langle \vec{v}, \vec{w}^{(n)} \rangle$.
Then
$$ \vec{w}^{(n)}
=
( \psi([L_n\vec{e}_1, n + 1]), \ldots, \psi([L_n\vec{e}_{d_n}, n + 1]) )
=
L_n^T \vec{w}^{(n + 1)}, $$
so~$(\vec{w}^{(n)})_{n = 1}^\infty$ is an element of~$\varprojlim ((\RR^+)^{d_n}, L_n^T )_{n = 1}^\infty$.
Our hypotheses implies that there is~$\lambda \ge 0$ such that $(\vec{w}^{(n)})_{n = 1}^\infty = (\lambda \vec{x}^{(n)})_{n = 1}^\infty$.
Thus~$\psi = \lambda \phi$.
\end{proof}

\subsection{The unital ordered group of a minimal Cantor system}
\label{ss:complete invariant}
In this section, as well as in the next section, a \emph{minimal Cantor system} is a pair~$(X, T)$, where~$X$ is a Cantor set and~$T : X \to X$ is a continuous and surjective map whose action on~$X$ is minimal.
If in addition~$T$ is a homeomorphism we call $(X, T)$ a \emph{homeomorphic minimal Cantor system}.\footnote{In the introduction we used ``minimal Cantor system'' to refer to what we call in this section, as well as in the next one, a ``homeomorphic minimal Cantor system''.}

We denote by $C(X,\ZZ)$ the space of all continuous functions defined on~$X$ and taking values in~$\ZZ$.
We will denote by~$\mathds{1}_X \in C(X, \ZZ)$ the constant function equal to~$1$ on~$X$.
Since~$X$ has a countable base of clopen sets, the space $C(X,\ZZ)$ is countable.
Thus $C(X,\ZZ)$ is a countable Abelian group with the usual
addition of functions.

The subgroup of \emph{coboundaries} of $C(X,\ZZ)$ is defined as
$$ \partial_TC(X,\ZZ)
\=
\left\{ f - f\circ T : f\in C(X,\ZZ) \right\}. $$
The quotient group $C(X,\ZZ)/\partial_TC(X,\ZZ)$ is denoted by
$K^0(X,T)$ and the subset of~$K^0(X, T)$ of all the equivalence classes represented by a non-negative function is denoted by~~$K^0(X, T)^+$.
The equivalence class in $C(X,\ZZ)/\partial_TC(X,\ZZ)$ of a function~$f$ in~$C(X, \ZZ)$ is denoted by~$[f]$.
Then the triple
 $$G(X,T)
\=
(K^0(X,T),K^0(X,T)^+,[\mathds{1}_X]),$$
is a unital ordered group, see \cite[Proposition~5.1]{HerPutSka92}.\footnote{This result is only stated in the case where~$T$ is a homeomorphism, but the proof works in the case when~$T$ is not injective.}
Below we show that~$G(X, T)$ is canonically isomorphic as a unital ordered group to the corresponding group associated to the natural extension of~$(X, T)$, see Proposition~\ref{p:coborde3} in~\S\ref{ss:dynamical ordered group}.
It thus follows from~\cite[Theorem~5.4]{HerPutSka92} that $G(X, T)$ is actually a unital dimension group.
We call~$G(X, T)$ \emph{the unital ordered group associated to $(X,T)$}.
A dimension group is the unital ordered group associated to a (homeomorphic) minimal Cantor system if and only if it is acyclic and simple, see for example~\cite[Corollary~6.3]{HerPutSka92}.
The quotient group~$G(X, T)/ \inf G(X, T)$ has a natural structure of unital ordered group, see~\S\ref{ss:ordered-groups}, which we denote just by~$G(X, T)/ \inf G(X, T)$.

It follows from the definitions that if~$(X, T)$ and~$(X', T')$ are minimal Cantor systems and~$h: X \to X'$ is a homeomorphism such that~$h \circ T = T' \circ h$, then the linear map from $C(X, \ZZ)$ to~$C(X', \ZZ)$ defined by~$f \mapsto f \circ h$ induces an isomorphism of unital ordered groups between~$G(X, T)$ and~$G(X', T')$.

Giordano, Putnam and Skau show in \cite[Theorem~2.2]{GioPutSka95} that for a homeomorphic minimal Cantor system the group $G(X,T)/\inf(G(X,T))$ determines the orbit equivalence class of~$(X,T)$.
\subsection{Unital ordered groups and natural extensions}
\label{ss:dynamical ordered group}
Let $(X,T)$ be a minimal Cantor system. The \emph{natural
extension} of $(X,T)$ is the topological dynamical system
$(\widehat{X}, \widehat{T})$ given by
$$
\widehat{X}
= \left\{ (x_i)_{i = 0}^\infty \in X^{\NN \cup \{ 0 \}}: x_i=T(x_{i+1}), \text{ for every } i\in \NN \cup \{ 0 \} \right\},
$$
and $\widehat{T} \left( (x_i)_{i = 0}^\infty \right) = (T(x_0),x_0,x_1,\cdots)$.
Observe that $(\widehat{X},\widehat{T})$ is a homeomorphic minimal Cantor system and that the projection on the first coordinate $\pi:\widehat{X}\to X$ is a factor map.

The rest of this section is devoted to the proof of the following proposition.
\begin{proposition}
\label{p:coborde3}
Let~$(X, T)$ be a minimal Cantor system and let~$(\widehat{X}, \widehat{T})$ be its natural extension. 
Then the groups~$G(X, T)$ and~$G(\widehat{X}, \widehat{T})$ are isomorphic as unital ordered groups.
\end{proposition}
To prove this proposition consider the following subgroup of~$C(\widehat{X}, \ZZ)$,
$$ C_{\pi}(\widehat{X},\ZZ)
\=
\left\{\widehat{f} \in C(\widehat{X},\ZZ): \text{ there is $f \in C(X, \ZZ)$ such that } \widehat{f} = f \circ \pi \right\}. $$
\begin{lemma}
\label{l:coborde1}
For every $\widehat{f} \in C(\widehat{X},\ZZ)$ there exists $k\geq 0$ such that $\widehat{f} \circ \widehat{T}^k \in C_{\pi}(\widehat{X},\ZZ)$.
\end{lemma}
\begin{proof}
Let $k \geq 1$ be an integer such that for each~$x$ and~$y$ in $\widehat{X}$ whose first~$k$ coordinates coincide, we have~$\widehat{f}(x) = \widehat{f}(y)$.
Then for each~$z \in X$ and $(x_i)_{i = 0}^\infty$ and $(y_i)_{i = 0}^\infty$ in~$\pi^{-1}(z)$ we have
\begin{eqnarray*}
\widehat{f}\circ \widehat{T}^k(x) & = &\widehat{f}(T^k(z),T^{k-1}(z),\cdots,
z,x_1,x_2,\cdots )\\
 &= & \widehat{f}(T^k(z),T^{k-1}(z),\cdots,
z,y_1,y_2,\cdots)\\
 & = & \widehat{f}\circ \widehat{T}^k(y).
\end{eqnarray*}
\end{proof}
\begin{proof}[Proof of Proposition~\ref{p:coborde3}]
The linear map from~$C(X, T)$ to~$C(\widehat{X}, \widehat{T})$ defined by~$f \mapsto f \circ \pi$ induces an injective morphism~$\iota$ of unital ordered groups with the property that for $a \in K^0(X, T)$ we have $\iota(a) \in K^0(\widehat{X}, \widehat{T})^+$ if and only if~$a \in K^0(X, T)^+$, see~\cite[Proposition 3.1]{GlaWei95} and~\cite[\S14]{GotHed55}.
So we just need to prove that this morphism is surjective.
Let~$\widehat{f} \in C(\widehat{X}, \ZZ)$ be given and let~$k \ge 1$ be given by Lemma~\ref{l:coborde1}, so that $\left[ \widehat{f} \circ \widehat{T}^k \right]$ is in the image of~$\iota$.
Since
$$ \left[ \widehat{f} \circ \widehat{T}^k \right]
=
\left[ \widehat{f} \circ \widehat{T}^k + (\widehat{f} - \widehat{f} \circ \widehat{T}^k) \right]
=
\left[ \widehat{f} \right], $$
it follows that~$\left[ \widehat{f} \right]$ is also in the image of~$\iota$.
\end{proof}

\subsection{Bratteli diagrams and Bratteli-Vershik systems}
\label{ss:Bratteli diagrams}
In this section we briefly recall the concepts of Bratteli diagram and Bratteli-Vershik system.
For more details we refer to \cite{DurHosSka99} and \cite{HerPutSka92}.

A \textit{Bratteli diagram} is an infinite directed graph $(V,E)$,
such that the vertex set~$V$ and the edge set~$E$ can be
partitioned into finite sets
$$
V = V_0 \sqcup V_1 \sqcup  \cdots
\text{ and }
E = E_1 \sqcup E_2 \sqcup \cdots
$$
with the following properties:
\begin{itemize}
\item $V_0=\{v_0\}$ is a singleton.

\item For every $j \ge 1$, each edge in~$E_j$ starts in a vertex
in~$V_{j - 1}$ and arrives to a vertex in~$V_{j}$.

\item
Each vertex in~$V$ has at least one edge starting from it and each vertex different from~$v_0$ has at least one edge arriving to it.
\end{itemize}
Given~$n \in \NN \cup \{ 0 \}$ the~$n$-th \emph{transition} or \emph{incidence matrix} of the Bratteli diagram $B=(V,E)$ is the~$V_n \times V_{n + 1}$ matrix defined by
$$ M_n(v,v')
=
\text{ number of edges from } v \in V_n \text{ to } v' \in V_{n+1}.
$$

For a vertex $e \in E$ we denote by~$s(e)$ the vertex
where~$e$ starts and by~$r(e)$ the vertex to which~$e$ arrives. A
\textit{path} in $(V, E)$ is by definition a finite (resp.
infinite) sequence $e_1e_2 \ldots e_j$ (resp. $e_1e_2 ...$) such
that for each $\ell = 1, \ldots, j - 1$ (resp. $\ell = 1, \ldots$)
we have $r(e_\ell) = s(e_{\ell + 1})$. Note that for each
vertex~$v$ distinct from $v_0$ there is at least one path starting
at~$v_0$ and arriving to~$v$.

An \textit{ordered Bratteli diagram} $(V,E,\geq)$ is a Bratteli
diagram $(V,E)$ together with a partial order~$\geq$ on~$E$, so
that two edges are comparable if and only if they arrive at the
same vertex. For each $j \ge 1$ and $v \in V_j$ the partial
order~$\ge$ induces an order on the set of paths from~$v_0$ to~$V$
as follows:
$$
e_1\cdots e_j > f_1\cdots f_j
$$
if and only if there exists $j_0 \in \{1,\cdots, j \}$ such that
$e_{j_0} > f_{j_0}$ and such that for each $\ell \in \{ j_0 + 1,
\ldots, j \}$ we have $e_\ell = f_\ell$.

We say that an edge~$e$ is \textit{maximal} (resp.
\textit{minimal}) if it is maximal (resp. minimal) with respect to
the order~$\ge$ on the set of all edges in~$E$ arriving at~$r(e)$.

Fix an ordered Bratteli diagram $B \=(V,E,\geq)$. We denote
by~$X_B$ the set of all infinite paths in~$B$ starting at~$v_0$.
For a finite path $e_1 \ldots e_j$ starting at~$v_0$ we denote by
$U(e_1 \ldots e_j)$ the subset of~$X_B$ of all infinite paths
$e_1'e_2' \ldots$ such that for all $\ell \in \{ 1, \ldots, j \}$ we have $e_\ell' = e_\ell$.
We endow~$X_B$ with the topology generated by the sets $U(e_1 \ldots e_j)$.
Then each of this sets is clopen, so~$X_B$ becomes a compact Hausdorff space with a countable basis of clopen sets.

We denote by~$X_B^{\max}$ (resp. $X_B^{\min}$) the set of all
elements $(e_j)_{j = 1}^\infty$ of~$X_B$ so that for each $j \ge 1$ the
edge~$e_j$ is a maximal (resp. minimal). It is easy to see that
each of these sets is non-empty.

From now on we assume that the set~$X_B^{\min}$ is reduced to a
unique point, that we denote by~$x_{\min}$. We then
define the transformation $T_B:X_B\to X_B$ as follows:
\begin{itemize}
\item $T_B^{-1}(x_{\min}) = X_{\max}$.

\item Given $x \in X_B \setminus X_{\max}$, let $j \ge 1$ be the
smallest integer such that~$e_j$ is not maximal. Then we denote
by~$f_j$ the successor of~$e_j$ and by $f_1\ldots f_{j - 1}$ the
unique minimal path starting at~$v_0$ and arriving to~$s(f_k)$.
Then we put,
$$T_B(x)=f_1\cdots f_{k-1}f_ke_{k+1}e_{k+2}\ldots \ .$$
\end{itemize}
The map~$T_B$ is continuous, onto and invertible except
at~$x_{\min}$.
\begin{proposition}\label{p:Bratteli}
Let $B=(V,E,\geq)$ be a simple ordered Bratteli diagram such that
$X_B$ has only one minimal path and let~$(M_n)_{n = 0}^\infty$ be the sequence of transition matrices of~$B$.
Then the unital ordered group~$G(X_B, T_B)$ is isomorphic to the ordered group
$$ \varinjlim \left( (\ZZ^{V_n}, (\ZZ^{V_n})^+), M_n^T \right)_{n = 0}^\infty, $$
together with the unit~$[1, 0]$.
\end{proposition}
\begin{proof}
The proof is the same than for the case $B$ properly ordered, see for example~\cite[\S 6.6]{Dur10}.
\end{proof}

\section{A multidimensional Euclidean algorithm and finitely generated subgroups of~$\RR$}
\label{s:algorithm}
In this section we first define a multidimensional Euclidean algorithm in~\S\ref{ss:algorithm}, which is similar to the (homogeneous) Jacobi-Perron algorithm, see for example~\cite{Ber71,Sch73}.
In~\S\ref{ss:finitely generated} we show how this algorithm gives a canonical representation of a given finitely generated additive subgroup of~$\RR$ as a direct limit of certain non-negative matrices we call ``admissible''.
\subsection{A Jacobi-Perron type algorithm}
\label{ss:algorithm}
The input of the algorithm is an integer $d \ge 2$ and a strictly positive and non-increasing vector~$\vec{x}$ in~$\RR^d$.
The output is an integer~$d' \ge 1$ satisfying~$d' \le d$, a non-increasing vector~$\vec{x}'$ in~$\RR^{d'}$ such that~$x'_1 = x_d$, $x_1' > x_2'$ and~$x_{d'}' > 0$, together with a non-increasing vector $\vec{a} \in \NN^{d}$ satisfying~$a_{d} = 1$.
The output is uniquely determined by the existence of an injective function
$$ \sigma : \{ 1, \ldots, d' \} \to \{ 1, \ldots, d \} $$
such that~$\sigma(1) = d$ and such that, if we define the $d \times d'$ matrix~$A$ by
\begin{equation}
\label{e:admissible}
A(\cdot, j)
=
\begin{cases}
\vec{a} & \text{if } j = 1; \\
\sum_{k = 1}^{\sigma(j)} \vec{e}_k & \text{if } j \in \{ 2, \ldots, d' \};
\end{cases}
\end{equation}
then $\vec{x} = A \vec{x}'$, see Lemma~\ref{l:characteristic property} below.
In general the function~$\sigma$ is not uniquely determined by the input data.

To define the algorithm, we define as an intermediate step vectors~$\vec{y} \in \RR^d$ and $\vec{b} \in \NN^d$, as follows.
Put~$y_d \= x_d$, $b_d \= 1$ and for $j \in \{1, \ldots, d - 1 \}$ put
$$ y_j \= x_d \left\{ \frac{x_j - x_{j + 1}}{x_d} \right\}
\text{ and }
b_j \= \left[ \frac{x_j - x_{j + 1}}{x_d} \right]. $$
Note that $y_j \in [0, x_d)$ and $x_j - x_{j + 1} = y_j + b_j x_d$.

The vector~$\vec{a}$ is defined for $j \in \{1, \ldots, d - 1 \}$ by
$$ a_j = \sum_{k = j}^d b_k $$
and by~$a_d = 1$.
By definition we have
\begin{equation}
\label{e:intermediate}
\vec{x}
=
\begin{pmatrix}
1 & 1 & \dots & 1 & a_1
\\
0 & 1 & \dots & 1 & a_2
\\
\vdots & \vdots & \ddots & \vdots & \vdots
\\
0 & 0 & \dots & 1 & a_{d - 1}
\\
0 & 0 & \dots & 0 & 1
\end{pmatrix}
\vec{y}.
\end{equation}
Then $d' \= \# \{ j \in \{1, \ldots, d \} : y_j > 0 \}$ and we consider an injective function
$$ \sigma : \{ 1, \ldots, d' \} \to \{ 1, \ldots, d \} $$
so that for each~$j \in \{ 1, \ldots, d' \}$ we have $y_{\sigma(j)} > 0$ and so that the number~$y_{\sigma(j)}$ is non-decreasing with~$j$.
Then the vector
$$ \vec{x}' \= (y_{\sigma(j)})_{j = 1}^{d'} $$
is independent of the choice of~$\sigma$.
By definition~$1 \le d' \le d$ and the vector~$\vec{x}'$ satisfies~$x'_1 = x_d$, $x_1' > x_2'$ and~$x_{d'} > 0$.
Furthermore, $\sigma(1) = d$ and if~$A$ is the matrix defined by~\eqref{e:intermediate}, then $\vec{x} = A \vec{x}'$.

In the following simple lemma we capture one of the properties of the algorithm that will be very important for the representation of countable additive subgroups of~$\RR$.
\begin{lemma}
\label{l:generating property}
Let~$d \ge 2$ be an integer and let~$\vec{x}$ be a strictly positive and non-increasing vector in~$\RR^d$.
Let~$d' \ge 1$ and~$\vec{x}'$ be the corresponding integer and the corresponding vector given by algorithm defined in~\S\ref{ss:algorithm}.
Then the additive subgroups of~$\RR$ generated by the coordinates of~$\vec{x}$ and by those of~$\vec{x}'$ coincide.
\end{lemma}
\begin{proof}
Since~$\vec{x}$ and the vector~$\vec{y} \in \RR^d$ determined by~\eqref{e:intermediate} are related by a unimodular matrix, the additive subgroups of~$\RR$ generated by the coordinates of~$\vec{x}$ and by those of~$\vec{y}$ coincide.
Thus the desired assertion follows from the fact that the vectors~$\vec{y}$ and~$\vec{x}'$ have the same non-zero coordinates.
\end{proof}

Given integers~$d \ge 2$ and~$d' \ge 1$ satisfying $d' \le d$, a non-increasing vector~$\vec{a} \in \NN^{d}$ and an injective function~$\sigma : \{ 1, \ldots, d' \} \to \{ 1, \ldots, d \}$ such that $\sigma(1) = d$, we denote by~$A(\vec{a}, \sigma)$ the $d' \times d$ matrix~$A$ defined by~\eqref{e:admissible}.
We call such a matrix \emph{admissible}.
Note that a square admissible matrix is unimodular.
\begin{lemma}
\label{l:characteristic property}
Let~$d \ge 2$ be an integer and let~$\vec{x}$ be a strictly positive and non-increasing vector in~$\RR^d$.
Then there is a unique integer~$d' \ge 1$, a unique non-increasing vector~$\vec{x}' \in \RR^{d'}$ and a unique non-increasing vector~$\vec{a} \in \ZZ^d$ such that
$$ d' \le d,
x'_1 = x_d,
x_1' > x_2',
x_{d'} > 0,
a_d = 1 $$
and such that there is an injective function~$\sigma : \{ 1, \ldots, d' \} \to \{ 1, \ldots, d \}$ such that~$\sigma(1) = d$ and
$$ \vec{x} = A(\vec{a}, \sigma) \vec{x}'. $$
If furthermore the coordinates of~$\vec{x}$ are rationally independent, then~$\vec{x}$ is strictly decreasing, $d' = d$, the coordinates of~$\vec{x}'$ are rationally independent and~$\sigma$ is uniquely determined by~$\vec{x}$.
\end{lemma}
\begin{proof}
The existence of~$d'$, $\vec{x}'$ and~$\vec{a}$ is given by the above.

To prove uniqueness, let~$d', \vec{x}', \vec{a}$ and~$\sigma$ be as in the statement of the lemma and let~$\vec{y} \in \RR^d$ be the vector defined for~$j \in \{1, \ldots, d \}$ by
$$ y_j
=
\begin{cases}
0 & \text{if } j \not\in \sigma(\{1, \ldots, d' \}); \\
x'_{\sigma^{-1}(j)} & \text{otherwise}.
\end{cases}
$$
We clearly have~\eqref{e:intermediate}.
Thus~$\vec{y}$, and hence~$d'$ and $\vec{x}'$, are uniquely determined by~$d$, $\vec{x}$ and~$\vec{a}$.
It remains to show that~$\vec{a}$ is uniquely determined by~$d$ and~$\vec{x}$.
To see this, observe that by~\eqref{e:intermediate} for each~$j \in \{1, \ldots, d - 1 \}$ we have
$$ x_j - x_{j + 1} = y_j + y_d(a_j - a_{j + 1}). $$
Since~$y_d = x_1' > x_2' \ge y_j \ge 0$ and~$a_j - a_{j + 1}$ is an integer, it follows that $a_j - a_{j + 1}$ is uniquely determined by~$x_j - x_{j + 1}$.
Since this holds for each $j \in \{ 1, \ldots, d - 1 \}$ and by definition~$a_d = 1$, it follows that~$\vec{a}$ is uniquely determined by~$d$ and~$\vec{x}$.

Suppose that the coordinates of~$\vec{x}$ are rationally independent and let~$\Gamma$ be the additive subgroup of~$\RR$ generated by the coordinates of~$\RR$.
Then the coordinates of~$\vec{x}$ form a base of~$\Gamma$ and the rank of~$\Gamma$ is equal to~$d$.
By Lemma~\ref{l:generating property} the coordinates of~$\vec{x}' \in \RR^{d'}$ generate~$\Gamma$.
Since~$d' \le d$, we conclude that~$d' = d$ and that the coordinates of~$\vec{x}'$ form a base of~$\Gamma$ and are thus rationally independent.
In particular, the coordinates of~$\vec{x}'$ are pairwise distinct.
So~$\sigma$ is characterized as the unique permutation of~$\{1, \ldots, d \}$ such that~$j \mapsto y_{\sigma(j)}$ is decreasing with~$j$.
\end{proof}

\subsection{Representing finitely generated additive subgroups of~$\RR$}
\label{ss:finitely generated}
Let~$d \ge 2$ be an integer and let~$\Gamma$ be a finitely generated additive subgroup of~$\RR$ of rank~$d$.
Given a non-increasing vector~$\vec{x}^{(0)}$ in~$\RR^d$ whose coordinates are strictly positive and form a base~$\Gamma$, we define for each~$n \in \NN$ the following objects recursively:
\begin{itemize}
\item
a strictly decreasing vector~$\vec{x}^{(n)}$ in~$\RR^d$ whose coordinates are strictly positive, rationally independent and form a base of~$\Gamma$;
\item 
a non-increasing vector~$\vec{a}^{(n)}$ in~$\NN^d$ such that~$a^{(n)}_d = 1$;
\item 
a permutation~$\sigma_n$ of~$\{1, \ldots, d \}$ such that~$\sigma_n(1) = d$ and
$$ \vec{x}^{(n - 1)} = A(\vec{a}^{(n)}, \sigma_n) \vec{x}^{(n)}. $$
\end{itemize}
Once~$\vec{x}^{(n - 1)}$ is defined, let~$d'$, $\vec{x}'$, $\vec{a}$ and~$\sigma$ be given by the algorithm defined in~\S\ref{ss:algorithm} with~$\vec{x} = \vec{x}^{(n - 1)}$.
Then Lemma~\ref{l:characteristic property} implies that~$d' = d$ and, together with Lemma~\ref{l:generating property}, that $\vec{x}^{(n)} \= \vec{x}'$, $\vec{a}^{(n)} \= \vec{a}$ and~$\sigma_n \= \sigma$ have the properties above.
Moreover, the vectors $\vec{x}^{(n)}$ and $\vec{a}^{(n)}$ and the permutation~$\sigma_n$ are all uniquely determined by~$\vec{x}^{(n - 1)}$ and hence by~$\vec{x}^{(0)}$.
\begin{proposition}
\label{p:finitely generated}
Let~$\Gamma$ be a finitely generated additive subgroup of~$\RR$ containing~$1$ and of rank~$d \ge 2$ and let~$\vec{x}^{(0)} \in \RR^d$ be a non-increasing vector whose coordinates are strictly positive and form a base of~$\Gamma$.
Furthermore, let~$(\vec{x}^{(n)})_{n = 1}^{\infty}$, $(\vec{a}^{(n)})_{n = 1}^\infty$ and~$(\sigma_n)_{n = 1}^\infty$ be given by the iteration of the algorithm defined in~\S\ref{ss:algorithm}, as above, and consider the ordered group
$$ (G, G^+)
\=
\varinjlim \left( \left( \ZZ^d, (\ZZ^d)^+ \right), A \left( \vec{a}^{(n)}, \sigma_n \right)^T \right)_{n = 1}^\infty. $$
Then there is a function~$\phi : G \to \RR$ such that for each~$[\vec{v}, n] \in G$ we have $\phi([\vec{v}, n]) = \langle \vec{v}, \vec{x}^{(n)} \rangle$ and such that it takes values in~$\Gamma$ and is an isomorphism between~$(G, G^+)$ and~$(\Gamma, \Gamma \cap \RR^+)$.
\end{proposition}
To prove this proposition for each integer~$d \ge 2$ we denote by~$\interior(\RR^+)^d$ the interior of~$(\RR^+)^d$, which is the cone of strictly positive vectors in~$\RR^d$.
We define the \emph{Hilbert projective metric~$\Theta_d$} on~$\interior(\RR^+)^d$ by
$$ \Theta_d(\vec{x}, \vec{x}')
\=
\ln \left( \frac{\inf \{ \mu > 0 : \mu \vec{x} \ge \vec{x}' \}}{\sup \{\lambda > 0 : \lambda \vec{x} \le \vec{x}'\}} \right). $$
Note that~$\Theta_d(\vec{x}, \vec{x}') = 0$ if and only if~$\vec{x}$ and~$\vec{x}'$ are proportional.
On the other hand, for an integer $d' \ge 2$ and a strictly positive $d' \times d$ matrix~$A$ we have
$$ A((\RR^+)^{d'} \setminus \{ \vec{0} \})
\subset
\interior(\RR^+)^d $$
and
$$ D(A)
\=
\sup \left\{ \Theta_d( A\vec{y}, A\vec{y}') : \vec{y}, \vec{y}' \in (\RR^+)^{d'} \setminus \{ \vec{0} \} \right\} < \infty. $$
Then for every~$\vec{y}, \vec{y}' \in (\RR^+)^{d'} \setminus \{ \vec{0} \}$ we have
\begin{equation*}
\Theta_d( A\vec{y}, A\vec{y}')
\le
\tanh(D(A)/4) \cdot \Theta_{d'}( \vec{y}, \vec{y}'),
\end{equation*}
see for example the proof of~\cite[Lemma~1 in~\S4]{Bir57} or~\cite[Theorem~1.1]{Liv95}.
Thus, if in addition~$d'' \ge 2$ is an integer and~$A'$ is a strictly positive~$d'' \times d'$ matrix, then
\begin{equation}
\label{e:Schwarz-Birkhoff}
D(AA') \le \tanh(D(A)/4) \cdot D(A').
\end{equation}

The following lemma is a direct consequence of the previous considerations.
\begin{lemma}
\label{l:showing unique ergodicity}
Let~$(d_n)_{n = 1}^\infty$ be a sequence of integers in~$\NN$ and for each~$n \in \NN$ let~$A_n$ be a strictly positive $d_n \times d_{n + 1}$ matrix.
If there is~$D_0 > 0$ such that for every~$n \in \NN$ we have~$D(A_n) \le D_0$, then for each~$n \in \NN$ there is~$\vec{x}^{(n)} \in (\RR^+)^{d_n}$ such that
$$ \varprojlim ((\RR^+)^{d_n}, A_n)_{n = 1}^\infty
=
\left\{ (\lambda \vec{x}^{(n)})_{n = 1}^\infty : \lambda \ge 0 \right\}. $$
\end{lemma}
The following lemma is needed for the proof of Proposition~\ref{p:finitely generated}.
\begin{lemma}
  \label{l:definite contraction}
Given an integer~$d \ge 2$, for every pair of $d \times d$ admissible matrices~$A$ and~$A'$ the matrix~$AA'$ is strictly positive and
$$ D(AA') \le 2 \ln d. $$
\end{lemma}
\begin{proof}
For each~$j \in \{2, \ldots, d\}$ we have~$A(\cdot, j) \le A(\cdot, 1)$ and therefore
$$ A(\cdot, 1) \le AA'(\cdot, j) \le dA(\cdot, 1). $$
This implies in particular that the matrix~$AA'$ is strictly positive.
Moreover, if~$a' \= \langle A'(\cdot, 1), \vec{e}_1 \rangle$ is the first coordinate of~$A'(\cdot, 1)$, then 
$$ a' A(\cdot, 1) \le AA'(\cdot, 1) \le d a' A(\cdot, 1). $$
So for~$\vec{v} = (v_1, \ldots, v_d) \in (\RR^+)^d \setminus \{ \vec{0} \}$ we have
$$ (a' v_1 + v_2 + \cdots + v_d) A(\cdot, 1)
\le
AA' \vec{v}
\le d (a' v_1 + v_2 + \cdots + v_d) A(\cdot, 1). $$
This implies  $\Theta_d(AA'\vec{v}, A(\cdot, 1)) \le \ln d$ and using the triangular inequality we get~$D(AA') \le 2 \ln d$.
\end{proof}
\begin{proof}[Proof of Proposition~\ref{p:finitely generated}]
Let~$(\vec{x}^{(n)})_{n = 1}^\infty$ be the sequence defined from~$\vec{x}^{(0)}$, as in the beginning of this section.
It is by definition an element of
$$ \varprojlim \left( (\RR^+)^d, A\left( \vec{a}^{(n)}, \sigma_n \right) \right)_{n = 1}^\infty. $$
On the other hand, Lemma~\ref{l:showing unique ergodicity} and Lemma~\ref{l:definite contraction} imply that this inverse limit is equal to $\{ (\lambda \vec{x}^{(n)})_{n = 1}^\infty : \lambda \ge 0 \}$.
So the hypotheses of Lemma~\ref{l:nexo} are satisfied with
$$ (d_n)_{n = 1}^\infty = (d)_{n = 1}^\infty
\text{ and } (L_n)_{n = 1}^\infty = \left( A \left( \vec{a}^{(n)}, \sigma_n \right)^T \right)_{n = 1}^\infty. $$
Let~$\phi$ be function given by this lemma, which is the unique state of~$(G, G^+)$.
Since for each~$n \in \NN$ the coordinates of~$\vec{x}^{(n)}$ are rationally independent, it follows that~$\phi$ is injective and hence that the infinitesimal group of~$(G, G^+)$ is trivial.
Since furthermore for each~$n \in \NN$ the coordinates of~$\vec{x}^{(n)}$ form a base of~$\Gamma$, by Lemma~\ref{l:nexo} it follows that~$(G, G^+)$ is isomorphic to~$(\Gamma, \Gamma \cap \RR^+)$.
\end{proof}

\section{Dimension groups with a unique state}
\label{s:infinitely generated}
We start this section by introducing a class of matrices we call ``basic'', that appear naturally as transition matrices of the Bratteli-Vershik system associated to the generalized odometer associated to a unimodal map, see~\S\ref{ss:generalized odometers}.
The rest of the section is devoted to represent a given countable additive subgroup of~$\RR$ containing~$1$, that is not contained in~$\QQ$, as a direct limit of basic matrices (Proposition~\ref{p:infinitely generated}).

Let~$V$ be a non-empty finite subset of~$\NN \cup \{ 0 \}$, let~$v$ be its minimal element and suppose that~$v + 1 \in V$.
Let~$V'$ be a finite subset of~$\NN$ containing~$V \setminus \{ v \}$ and whose minimal element is~$v + 1$.
Then we denote by~$B(V, V')$ the~$V \times V'$ matrix defined for~$j \in V'$ by
$$ B(V, V')( \cdot, j) =
\begin{cases}
\vec{e}_v + \vec{e}_{v + 1} & \text{if } j = v + 1; \\
\vec{e}_j & \text{if } j \in V \setminus \{v, v + 1 \}; \\
\vec{e}_v & \text{if } j \in V' \setminus V.
\end{cases} $$

The purpose of this section is to prove the following proposition.
\begin{proposition}
\label{p:infinitely generated}
Let~$\Gamma$ be a countable additive subgroup of~$\RR$ containing~$1$, but not contained in~$\QQ$.
Then for each~$j \in \NN$ there is a finite subset~$V_j$ of~$\NN$ such that $V_1 = \{ 1, 2 \}$ and such that the following properties hold:
\begin{itemize}
\item[1.]
there is a strictly increasing sequence of integers~$(q_n)_{n = 1}^\infty$ such that~$q_1 = 1$, $q_2 = 3$ and such that for each~$n \in \NN$ and~$j \in [q_{n}, q_{n + 1} - 1]$ we have
$$ V_{j + 1} \setminus V_j \subset [q_{n + 1}, q_{n + 2} - 1]; $$
\item[2.]
for each~$j \in \NN$ the basic matrix~$B(V_j, V_{j + 1})$ is defined and the ordered group
$$ (G, G^+)
\=
\varinjlim \left( \left(\ZZ^{V_j}, (\ZZ^{V_j})^+ \right), B(V_j, V_{j + 1})^T \right)_{j = 1}^\infty $$
together with the order unit~$\left[ \bigl( \begin{smallmatrix} 1 \\ 1 \end{smallmatrix} \bigr), 1 \right]$ is isomorphic to~$(\Gamma, \Gamma \cap \RR^+, 1)$.
\end{itemize}
\end{proposition}

Roughly speaking, the idea of the proof of this proposition is to use recursively the algorithm defined in~\S\ref{ss:algorithm} as in Proposition~\ref{p:finitely generated}, to obtain a sequence of bases of finitely generated subgroups of~$\Gamma$ whose union is equal to~$\Gamma$. 
The inductive step is given by Lemma~\ref{l:augmenting generator} and in Lemma~\ref{l:basic generate} we prove that each of the matrices used in this process can be written as a product of basic matrices.

Given integers~$d \ge 2$ and $d' \ge 1$ we say that a $d \times d'$ matrix $M$ is \emph{strictly decreasing} if each of its columns is and if for each $j \in \{ 1, \ldots, d' - 1 \}$ the vector $M(\cdot, j) - M(\cdot, j + 1)$ is also strictly decreasing.
Furthermore, we will say that~$M$ is \emph{strictly decreasing up to a permutation} if there is a permutation~$\tau$ of~$\{1, \ldots, d' \}$ fixing~$1$ and such that the matrix~$M'$ defined for $j \in \{1, \ldots, d' \}$ by $M'(\cdot, j) = M(\cdot, \tau(j))$ is strictly decreasing.

\begin{lemma}
\label{l:augmenting generator}
Let~$\Gamma'$ be a finitely generated additive subgroup of~$\RR$ containing~1 and of rank $d' \ge 2$.
Given a non-increasing vector~$\vec{x}^{(0)} \in \RR^{d'}$ whose coordinates are strictly positive and form a base of~$\Gamma'$, let~$(\vec{x}^{(n)})_{n = 1}^\infty$ be given by the iteration of the algorithm defined in~\S\ref{ss:algorithm}, as in~\S\ref{ss:finitely generated} with~$\Gamma = \Gamma'$ and~$d' = d$.
Then the following properties hold.
\begin{enumerate}
\item[1.]
Let~$y$ be a strictly positive element of~$\Gamma'$ and for each~$n \in \NN$ let~$\vec{v}^{(n)}$ be the unique vector in~$\ZZ^{d'}$ such that~$\langle \vec{v}^{(n)}, \vec{x}^{(n)} \rangle = y$.
Then for every sufficiently large~$n$ the vector~$\vec{v}^{(n)}$ has strictly positive coordinates, the first being the largest.
\item[2.]
Let~$d \ge 2$ be an integer and let~$\vec{y}$ be a strictly decreasing vector in~$\RR^{d}$ having all of its coordinates strictly positive and in~$\Gamma'$.
For each~$n \in \NN$ let~$M_n$ be the unique integer $d \times d'$ matrix such that~$\vec{y} = M_n \vec{x}^{(n)}$.
Then for every sufficiently large integer~$n$ the matrix~$M_n$ is strictly positive and strictly decreasing up to a permutation and satisfies~$M_n(1, 1) \ge 5$ and $D(M_n) \le 1$.
\end{enumerate}
\end{lemma}
\begin{proof}
For each~$n \in \NN$ let~$\vec{a}^{(n)}$ and~$\sigma_n$ be given by the iteration of the algorithm defined in~\S\ref{ss:algorithm}, as in~\S\ref{ss:finitely generated} with~$\Gamma = \Gamma'$ and~$d = d'$.

\partn{1}
Let~$y \in \Gamma'$ be such that~$y > 0$ and for each~$n \in \NN$ let~$\vec{v}^{(n)}$ be such that~$\langle \vec{v}^{(n)}, \vec{x}^{(n)} \rangle = y$.
Let~$\phi$ be given by Proposition~\ref{p:finitely generated}, so that for every~$n \in \NN$ we have~$\phi( [\vec{v}^{(n)}, n]) = y$.
Since~$y > 0$, for every sufficiently large~$m$ we have~$\vec{v}^{(m)} \in (\ZZ^{d'})^+$.
It follows that for such~$m$ the first coordinate of~$\vec{v}^{(m + 1)} = A(\vec{a}^{(m)}, \sigma_{m})^T \vec{v}^{(m)}$ is strictly positive and that~$\vec{v}^{(m + 2)} = A(\vec{a}^{(m + 1)}, \sigma_{m + 1})^T \vec{v}^{(m + 1)}$ has all of its coordinates strictly positive, the first being the largest.
This proves the desired assertion.

\partn{2}
By Part~1 for every sufficiently large~$m$ and each $j \in \{1, \ldots, {d} \}$, the unique vector~$\vec{v} \in \ZZ^{d'}$ such that~$\langle \vec{v}, \vec{x}^{(m)} \rangle = y_j$ is strictly positive.
For each~$j \in \{1, \ldots, {d} - 1 \}$ the same holds for~$y_{j + 1} - y_j > 0$.
Thus there is~$m_0 \in \NN$ such that for every integer~$m \ge m_0$ the matrix~$M_m$ is strictly positive and each of its columns is strictly decreasing.
We will prove now that for $m \ge m_0 + 1$ the matrix
$$ M_m = M_{m - 1} A (\vec{a}^{(m)}, \sigma_m) $$
is in addition strictly decreasing up to a permutation.
Let~$\tau$ be the permutation of~$\{1, \ldots, d' \}$ defined by~$\tau(j) = \sigma_m^{-1}(d' + 1 - j)$, so that~$\tau$ fixes~$1$.
Since each of the columns of~$M_{m - 1}$ is strictly decreasing, the vector
$$ M_m(\cdot, \tau(1)) - M_m(\cdot, \tau(2))
=
a^{(m)}_{d'} M_{m - 1}(\cdot, d') + \sum_{j = 1}^{d' - 1} \left( a^{(m)}_j - 1 \right) M_{m - 1}(\cdot, j) $$
and, if~$d' \ge 3$, for each~$j \in \{2, \ldots, d' - 1 \}$ the vector
$$ M_m( \cdot, \tau(j)) - M_m(\cdot, \tau(j + 1))
=
M_{m - 1}(\cdot, d' + 1 - j)$$
is strictly decreasing.
This proves that the matrix~$M_m$ is strictly decreasing up to a permutation.
It remains to show that for large~$m$ we have~$M_m(1, 1) \ge 5$ and~$D(M_m) \le 1$.
Let~$m \ge m_0 + 1$ be given.
Since~$M_m$ is integer, strictly positive and strictly decreasing up to a permutation, we have~$M_m(1, 1) \ge d \cdot d' \ge 4$.
Thus
$$ M_{m + 1}(1, 1)
=
\sum_{j = 1}^{d'} \vec{a}^{(m + 1)}_j M_m(1, j) \ge M_m(1, 1) + M_m(1, d')
\ge
5. $$
On the other hand, applying repeatedly Lemma~\ref{l:definite contraction} and~\eqref{e:Schwarz-Birkhoff} we have for each integer~$\ell \ge 1$,
\begin{multline*}
D(M_{m + 2\ell + 1})
=
D \left( M_m A \left( \vec{a}^{(m + 1)}, \sigma_{m + 1} \right) \cdots A \left( \vec{a}^{(m + 2\ell + 1)}, \sigma_{m + 2 \ell + 1} \right) \right)
\\ \le
\left(\tanh(\ln d'/ 2) \right)^{\ell} (2 \ln d').
\end{multline*}
This proves the desired assertion and completes the proof of the lemma.
\end{proof}

Let~$V, V', \widetilde{V}, \widetilde{V}' \subset \NN \cup \{ 0 \}$ be finite sets such that $\# V = \# \widetilde{V}$ and $\# V' = \# \widetilde{V}'$ and let~$\tau : V \to \widetilde{V}$ and $\tau' : V' \to \widetilde{V}'$ be increasing bijections.
Then we say that a $V \times V'$ matrix~$M$ and a $\widetilde{V} \times \widetilde{V}'$ matrix~$\widetilde{M}$ are \emph{equal up to increasing re-indexing} if for each~$(j, j') \in V \times V'$ we have~$\widetilde{M}(j, j') = M(\tau(j), \tau(j'))$.


\begin{lemma}
\label{l:basic generate}
Let~$d, d' \ge 2$ be integers and let~$M$ be an integer and strictly positive $d \times d'$ matrix which is strictly decreasing up to a permutation.
Suppose furthermore that~$M(1, 1) \ge 5$, so that
$$ k \= d - 3 + M(1, 1) \ge d + 2. $$
Then there are finite subsets~$W_0, \ldots, W_{k}$ of~$\NN \cup \{ 0 \}$ with~$W_0 = [0, d - 1]$ and~$W_k = [k, k + d' - 1]$, such that the following properties hold.
\begin{enumerate}
\item[1.]
For~$j \in [0, d - 1]$ we have~$W_{j + 1} \setminus W_j \subset [d, k - 1]$ and for~$j \in [d, k - 1]$ we have~$W_{j + 1} \setminus W_j \subset [k, k + d' - 1]$.
\item[2.]
For each~$j \in [0, k - 1]$ the basic matrix~$B(W_j, W_{j + 1})$ is defined and the product
$$ B(W_0, W_1) \cdots B(W_{k - 1}, W_{k}) $$
is equal to~$M$ up to increasing re-indexing.
\end{enumerate}
\end{lemma}
Recall that for~$n, m \in \ZZ$ with~$n \ge m + 1$ we use~$[n, m]$ to denote the empty set.
\begin{proof}
Let~$\tau$ be the permutation of $\{1, \ldots, d' \}$ fixing~$1$ and such that the $d \times d'$ matrix~$M'$ defined for $t \in [1, d']$ by
$$ M'( \cdot, t) = M(\cdot, \tau(t)), $$
is strictly decreasing.

It will be convenient to index the column vectors of~$M$ and of~$M'$ on $[0, d - 1]$ instead of $[1, d]$.
So we define the $[0, d - 1] \times [1, d']$ matrices $\widetilde{M}$ and~$\widetilde{M}'$, for $(s, t) \in [0, d - 1] \times [1, d']$ by
$$ \widetilde{M}(s, t) = M(s + 1, t)
\text{ and }
\widetilde{M}'(s, t) = M'(s + 1, t). $$

After some preliminary definitions in Part~1, we define the sets~$W_j$ in Part~2 and then prove the desired properties in Part~3.

\partn{1}
The vector
$$ \vec{a}^{(1)} \= \widetilde{M}'(\cdot, d') = \widetilde{M}(\cdot, \tau(d')) $$
and, if~$d' \ge 3$, for each $t \in \{2, \ldots, d' - 1 \}$ the vector
\begin{multline*}
\vec{a}^{(t)}
\=
\widetilde{M}'(\cdot, d' + 1 - t) - \widetilde{M}'(\cdot, d' + 2 - t)
\\ =
\widetilde{M}(\cdot, \tau(d' + 1 - t)) - \widetilde{M}(\cdot, \tau(d' + 2 - t))
\end{multline*}
is strictly positive and strictly decreasing.
In particular, for each $t \in \{1, \ldots, d' - 2 \}$ we have $a_0^{(t)} \ge d \ge 2$.

On the other hand, the vector
$$ \widetilde{M}'(\cdot, 1) - \widetilde{M}'(\cdot, 2)
=
\widetilde{M}(\cdot, \tau(1)) - \widetilde{M}(\cdot, \tau(2)) $$
is also strictly positive and strictly decreasing.
So the vector,
$$ \vec{a}^{(d')}
\=
\widetilde{M}(\cdot, \tau(1)) - \widetilde{M}(\cdot, \tau(2)) - \left( 2 \vec{e}_0 + \sum_{s = 1}^{d - 1} \vec{e}_s \right) $$
is non-negative and non-increasing.
Note that
$$ \sum_{t = 1}^{d'} a_0^{(t)} = \widetilde{M}(0, 1) - 2. $$

\partn{2}
Put $W_0 \= [0, d - 1]$ and for $j \in [1, d - 1]$,
\begin{multline*}
W_j
\=
\left[ j, d - 1 + a_0^{(1)} - a_j^{(1)} \right]
\\ \cup
\bigcup_{t_0 = 2}^{d'} \left[ d - 1 + \sum_{t = 1}^{t_0 - 1} a_0^{(t)}, d - 2 + \sum_{t = 1}^{t_0} a_0^{(t)} - a_j^{(t)} \right].
\end{multline*}
Put also
$$ W_d
\=
\left[d, d - 4 + \widetilde{M}(0, 1) \right]
= [d, k - 1] $$
and if $a_0^{(1)} \ge 3$, then for $j \in \left[ d + 1, d - 2 + a_0^{(1)} \right]$ put
$$ W_j
\=
\left[ j, d - 3 + \widetilde{M}(0, 1) \right]
= [j, k]. $$
For $t_0 \in [2, d']$ and $j \in \left[ d - 1 + \sum_{t = 1}^{t_0 - 1} a_0^{(t)}, d - 2 + \sum_{t = 1}^{t_0} a_0^{(t)} \right]$ put
\begin{multline*}
W_j
\=
\left[j, d - 3 + \widetilde{M}(0, 1) \right] \cup
\left(d - 4 + \widetilde{M}(0, 1) + \tau([d' + 2 - t_0, d']) \right)
\\ =
[j, k] \cup \left( k - 1 + \tau([d' + 2 - t_0, d'] \right).
\end{multline*}
Finally, put~$W_k = [k, k + d' - 1]$ and note that~$W_{k - 1} = [k - 1, k + d' - 1]$.

\partn{3}
It is straightforward to check Part~1 of the lemma.
To prove Part~2, first observe that by definition for each~$j \in [0, k]$ the minimal element of~$W_j$ is~$j$.
Furthermore, for~$j \in [0, k - 1]$ we also have~$j + 1 \in W_j$ and~$W_j \setminus \{ j \} \subset W_{j + 1}$.
This shows that the basic matrix~$B(W_j, W_{j + 1})$ is defined.
To calculate the product of these matrices we start observing that by a direct computation we have for each $j \in [1, d - 1]$
\begin{multline*}
B(W_0, W_1) \cdots B(W_{j - 1}, W_j)(\cdot, \ell)
\\ =
\begin{cases}
\sum_{s = 0}^j \vec{e}_s & \text{if } \ell = j; \\
\vec{e}_\ell & \text{if } \ell \in [j + 1, d - 1]; \\
\sum_{s = 0}^{s_0 - 1} \vec{e}_s & \text{if } \ell \in W_{s_0} \setminus W_{s_0 - 1} \text{ and } s_0 \in [1, j];
\end{cases}
\end{multline*}
and hence that
\begin{multline}
  \label{e:first product}
B(W_0, W_1) \cdots B(W_{d - 1}, W_d)(\cdot, \ell)
\\ =
\begin{cases}
2 \vec{e}_0 + \sum_{s = 0}^{d - 1} \vec{e}_s & \text{if } \ell = d; \\
\vec{e}_0 & \text{if } \ell \in (W_1 \setminus W_0) \setminus \{ d \}; \\
\sum_{s = 0}^{s_0 - 1} \vec{e}_s & \text{if } \ell \in W_{s_0} \setminus W_{s_0 - 1} \text{ and } s_0 \in [2, d].
\end{cases}
\end{multline}

On the other hand, if~$a_0^{(1)} \ge 3$, then for $j \in [d + 1, d - 2 + a_0^{(1)}]$ we have
\begin{multline*}
B(W_d, W_{d + 1}) \cdots B(W_{j - 1}, W_j)(\cdot, \ell)
=
\begin{cases}
\sum_{s = d}^j \vec{e}_s & \text{if } \ell = j; \\
\vec{e}_\ell & \text{if } \ell \in [j + 1, k - 1]; \\
\vec{e}_d & \text{if } \ell = k.
\end{cases}
\end{multline*}
In all cases, for~$t_0 \in [2, d']$ and $j \in \left[ d - 1 + \sum_{t = 1}^{t_0 - 1} a_0^{(t)}, d - 2 + \sum_{t = 1}^{t_0} a_0^{(t)} \right]$ we have
\begin{multline*}
B(W_d, W_{d + 1}) \cdots B(W_{j - 1}, W_j)(\cdot, \ell)
\\ =
\begin{cases}
\sum_{s = d}^j \vec{e}_s & \text{if } \ell = j; \\
\vec{e}_\ell & \text{if } \ell \in [j + 1, k - 1]; \\
\vec{e}_d & \text{if } \ell = k; \\
\sum_{s = d}^{d - 2 + \sum_{t' = 1}^{t - 1} a_0^{(t')}} \vec{e}_s & \text{if } t \in [2, t_0] \text{ and } \ell = k - 1 + \tau(d' + 2 - t).
\end{cases}
\end{multline*}
Combining the above with~$t_0 = d'$ and~$j = d - 2 + \sum_{t = 1}^{d'} a^{(t)}_0 = k - 1$ and~\eqref{e:first product} we obtain,
\begin{multline*}
B(W_0, W_1) \cdots B(W_{k - 2}, W_{k - 1})(\cdot, \ell)
\\ =
\begin{cases}
\sum_{t = 1}^{d'} \vec{a}^{(t)} & \text{if } \ell = k - 1; \\
2 \vec{e}_0 + \sum_{s = 1}^{d - 1} \vec{e}_s & \text{if } \ell = k; \\
\sum_{t = 1}^{t_0 - 1} \vec{a}^{(t)} & \text{if } t_0 \in [2, d'] \text{ and } \ell = k - 1 + \tau(d' + 2 - t_0).
\end{cases}
\end{multline*}
Notice that
$$ \sum_{t = 1}^{d'} \vec{a}^{(t)}
=
\widetilde{M}( \cdot, 1) - \left( 2 \vec{e}_0 + \sum_{s = 1}^{d - 1} \vec{e}_s \right)
$$
and that for $t_0 \in [2, d']$ and $\ell = k - 1 + \tau(d' + 2 - t_0)$ we have
$$ \sum_{t = 1}^{t_0 - 1} \vec{a}^{(t)}
=
\widetilde{M}(\cdot, \tau(d' + 2 - t_0))
=
\widetilde{M}(\cdot, \ell - (k - 1)). $$
Thus
\begin{multline*}
B(W_0, W_1) \cdots B(W_{k - 2}, W_{k - 1})(\cdot, \ell)
\\ =
\begin{cases}
\widetilde{M}( \cdot, 1) - \left( 2 \vec{e}_0 + \sum_{s = 1}^{d - 1} \vec{e}_s \right) & \text{if } \ell = k - 1; \\
2 \vec{e}_0 + \sum_{s = 1}^{d - 1} \vec{e}_s & \text{if } \ell = k; \\
\widetilde{M}(\cdot, \ell - (k - 1)) & \text{if } \ell \in [k + 1, k + d' - 1].
\end{cases}
\end{multline*}
A direct computation then shows that for each~$\ell \in W_k = [k, k + d' - 1]$ we have
$$ B(W_0, W_1) \cdots B(W_{k - 1}, W_{k})(\cdot, \ell)
=
\widetilde{M}(\cdot, \ell - (k - 1)). $$
This proves that the matrix $B(W_0, W_1) \cdots B(W_{k - 1}, W_{k})$ is equal to~$\widetilde{M}$, and hence to~$M$, up to an increasing re-indexing.
\end{proof}

\begin{proof}[Proof of Proposition~\ref{p:infinitely generated}]
Let~$\alpha \in \Gamma$ be an irrational number in the interval $(0, 1/2)$ and put~$\alpha_0 = 1$ and $\alpha_1 \= \alpha$.
For an integer~$\ell \ge 2$ define inductively~$\alpha_\ell \in \Gamma$ in such a way that
the set $\{ \alpha_\ell : \ell \in \NN \cup \{ 0 \} \}$ generates~$\Gamma$.
For~$\ell \in \NN \cup \{ 0 \}$ let~$\Gamma_\ell$ be the additive subgroup of~$\RR$ generated by $\{ \alpha_0, \ldots, \alpha_\ell \}$ and let~$d_{\ell} \ge 1$ be the rank of~$\Gamma_\ell$.
So $d_0 = 1$, $d_1 = 2$ and the sequence~$(d_\ell)_{\ell = 0}^\infty$ is non-decreasing.

\partn{1}
For each~$\ell \in \NN \cup \{ 0 \}$ define recursively a strictly decreasing vector~$\vec{y}^{(\ell)}$ in~$\RR^{d_\ell}$ whose coordinates are strictly positive and form a base of~$\Gamma_\ell$ and an integer and  strictly positive $d_\ell \times d_{\ell + 1}$ matrix~$M_\ell$ such that
\begin{equation}
\label{e:inductive step}
\vec{y}^{(\ell)} = M_\ell \vec{y}^{(\ell + 1)}
\text{ and }
D(M_\ell) \le 1,
\end{equation}
as follows.
Put~$\vec{y}^{(0)} = (1)$, $\vec{y}^{(1)} = \bigl( \begin{smallmatrix} 1 - \alpha_1 \\ \alpha_1 \end{smallmatrix} \bigr)$ and $M_0 = \bigl( \begin{smallmatrix} 1 & 1 \end{smallmatrix} \bigr)$.
Let~$\ell \in \NN$ be given and assume~$\vec{y}^{(\ell)}$ is already defined.
Let~$\vec{x}^{(0)} \in \RR^{d_{\ell + 1}}$ be a strictly decreasing vector whose coordinates are strictly positive and form a base of~$\Gamma_{\ell + 1}$.
Then by Part~2 of Lemma~\ref{l:augmenting generator} with
$$ \Gamma' = \Gamma_{\ell + 1},
d' = d_{\ell + 1}, 
d = d_\ell
\text{ and } 
\vec{y} = \vec{y}^{(\ell)} $$
there is an integer~$n \in \NN$ and an integer and strictly positive $d_\ell \times d_{\ell + 1}$ matrix $M_{\ell + 1}$ which is strictly decreasing up to a permutation and satisfies~$\vec{y}^{(\ell)} = M_\ell \vec{x}^{(n)}$ and~$D(M_\ell) \le 1$.
Setting~$\vec{y}^{(\ell + 1)} \= \vec{x}^{(n)}$ we have~\eqref{e:inductive step}.
On the other hand, the coordinates of~$\vec{y}^{(\ell + 1)}$ are strictly positive and form a base of~$\Gamma_{\ell + 1}$.

By definition~$(\vec{y}^{(\ell)})_{\ell = 0}^\infty$ belongs to the inverse limit
$$ \varprojlim \left( (\RR^+)^{d_\ell}), M_\ell \right)_{\ell = 0}^\infty, $$
which by Lemma~\ref{l:showing unique ergodicity} it is equal to~$\left\{ (\lambda \vec{y}^{(\ell)})_{\ell = 0}^\infty : \lambda \ge 0 \right\}$.
Then Lemma~\ref{l:nexo} implies that the ordered group
$$ (G, G^+)
\=
\varinjlim \left( (\ZZ^{d_\ell}, (\ZZ^{d_\ell})^+), M_\ell^T \right)_{\ell = 0}^\infty $$
has a unique state~$\phi$ up to a scalar factor, such that for each~$[\vec{v}, \ell] \in G$ we have~$\phi([\vec{v}, \ell]) = \langle \vec{v}, \vec{y}^{(\ell)} \rangle$.
Furthermore, $\phi$ induces an isomorphism between~$(G / \inf(G), (G / \inf(G))^+)$ and $(\phi(G), \phi(G) \cap \RR^+)$.
On the other hand, since for each~$\ell \in \NN$ the coordinates of~$\vec{y}^{(\ell)}$ form a base of~$\Gamma_\ell$, they are rationally independent.
Thus~$\phi(G) = \Gamma$ and~$\phi$ is injective.
Hence the subgroup~$\inf(G)$ of~$G$ is trivial and~$\phi$ is an isomorphism of unital ordered groups between~$(G, G^+, [1, 0])$ and $(\Gamma, \Gamma \cap \RR^+, 1)$.

\partn{2}
By construction for each~$\ell \in \NN$ we have~$d_{\ell}, d_{\ell + 1} \ge 2$ and the matrix~$M_{\ell}$ is integer, strictly positive and strictly decreasing up to a permutation.
Furthermore, it satisfies~$M_{\ell}(1, 1) \ge 5$.
By Lemma~\ref{l:basic generate} we can define by induction a sequence of finite subsets~$(V_j)_{j = 1}^\infty$ of~$\NN$ with~$V_1 = \{ 1, 2 \}$ and a strictly increasing sequence of integers~$(t_\ell)_{\ell = 0}^\infty$ with~$t_1 = 1$ such that the following properties hold.
\begin{enumerate}
\item[1.]
For each~$\ell \in \NN$ we have~$t_{\ell + 1} - t_{\ell} \ge d_{\ell} + 1$ and~$V_{t_\ell} = [t_\ell, t_{\ell} + d_{\ell} - 1]$.
Furthermore, for each~$j \in [t_{\ell}, t_{\ell} + d_{\ell} - 1]$ we have~$V_{j + 1} \setminus V_j \subset [t_{\ell} + d_{\ell}, t_{\ell + 1} - 1]$ and for each~$j \in [t_{\ell} + d_{\ell}, t_{\ell + 1} - 1]$ we have~$V_{j + 1} \setminus V_j \subset [t_{\ell + 1}, t_{\ell + 1} + d_{\ell + 1} - 1]$.
\item[2.]
For each~$j \in \NN$ the basic matrix~$B(V_j, V_{j + 1})$ is defined and for each~$\ell \in \NN$ the product
$$ B(V_{t_\ell}, V_{t_\ell + 1}) \cdots B(V_{t_{\ell + 1} - 1}, V_{t_{\ell + 1}}) $$
is equal to~$M_\ell$ up to increasing re-indexing.
\end{enumerate}
Part~2 of the proposition follows from property~2 above and by what was proved in Part~1 of the proof.
Part~1 of the proposition follows easily from property~1 above, for the sequence~$(q_n)_{n = 1}^\infty$ defined for~$\ell \in \NN$ by~$q_{2\ell - 1} = t_{\ell}$ and~$q_{2 \ell} = t_{\ell} + d_{\ell}$.
\end{proof}

\section{Generalized odometers and postcritical sets}
\label{s:realization}
After recalling the definition of generalized odometers and Bratteli-Vershik systems associated to unimodal maps, we prove in~\S\ref{ss:generalized odometers} that each of the dimension groups constructed in~\S\ref{s:infinitely generated} can be realized as the unital ordered group of a generalized odometer associated to a unimodal map (Theorem~\ref{t:almost-fullness}).
Theorem~\ref{t:fullness} is an easy consequence of this fact.
In~\S\ref{ss:logistic} we show that the generalized odometer constructed in Theorem~\ref{t:almost-fullness} is an extension of the \pcs{} of a logistic map, in such a way that the inverse of the corresponding factor map is defined on the complement of the backward orbit of the critical point (Proposition~\ref{p:projection-to-pcs}).
We end this section with the proofs of Theorem~\ref{t:fullness} and Theorem~\ref{t:semi-equivalence}.
\subsection{Generalized odometers}
\label{ss:generalized odometers}
Let $Q : \NN \cup \{ 0 \} \to \NN \cup \{ 0 \}$ be a map such that~$Q(0) = 0$ and such that for each integer~$k \ge 1$ we have~$Q(k) \le k - 1$.
Let~$(S_k)_{k = 1}^\infty$ be the strictly increasing sequence of integers defined recursively by~$S_0 = 1$ and for $k \ge 1$ by~$S_k = S_{k - 1} + S_{Q(k)}$.

Let~$\Omega_Q$ be the set defined as
\begin{multline*}
\Omega_Q
\=
\{ (x_k)_{k = 0}^\infty \in \{ 0, 1 \}^{\NN \cup \{ 0 \}} : x_k = 1 \text{ implies that }
\\
\text{for } j \in [Q(k + 1), k - 1] \text{ we have } x_j = 0 \}.
\end{multline*}
For each non-negative integer~$n$ there is a unique sequence
$\expansion{n} \= (x_k)_{k =0}^\infty$ in~$\Omega_Q$ that has at most finitely many~$1$'s and~$\sum_{k \ge 0} x_k S_k = n$.
The sequence~$\expansion{n}$ is also characterized as the unique sequence in $\{0, 1 \}^{\NN \cup \{ 0 \}}$ with finitely many~$1$'s such that $\sum_{k \ge 0} x_k S_k = n$ and such that, if  $0 \le k_0 < \cdots < k_l$ are all the integers verifying $x_{k_0}= \cdots =x_{k_l}=1$, then $S_{k_l}\leq n \le S_{k_l+1} - 1$ and for each~$j \in [0, l - 1]$ we have
$$ S_{k_j}\leq n - \sum_{s={j+1}}^{l}S_{k_s} \le S_{k_j + 1} - 1. $$

When~$Q(k) \to \infty$ as~$k \to \infty$, the map defined on the subset~$ \{ \expansion{n}
: n \in \NN \}$ of~$\Omega_Q$ by $\expansion{n} \mapsto
\expansion{n + 1}$ extends continuously to a map~$T_Q : \Omega_Q
\to \Omega_Q$ which is onto, minimal and such that~$T_Q^{-1}$ is
well defined on $\Omega_Q \setminus \expansion{0}$;
see~\cite[Lemma~2]{BruKelStP97}.
We call $(\Omega_Q, T_Q)$ the \textit{generalized odometer} associated to~$Q$.

The following definition will be important in what follows.
\begin{definition}
We say that a map $Q:\NN \cup \{ 0 \} \to \NN \cup \{ 0 \}$ such that~$Q(0) = 0$ is \emph{increasing modulo
intervals}, if there exist a strictly increasing sequence of integers~$(q_n)_{n = 0}^\infty$ such that $q_0 = 0$, $q_1 = 1$ and for each~$n \in \NN$
$$ Q([q_n, q_{n + 1} - 1]) \subset [q_{n - 1}, q_n - 1]. $$
\end{definition}
If~$Q : \NN \cup \{ 0 \} \to \NN \cup \{ 0 \}$ is increasing modulo intervals, then~$Q(0) = 0$, for every integer~$k \ge 1$ we have~$Q(k) \le k - 1$ and $Q(k) \to \infty$ as $k \to \infty$.
So the generalized odometer~$(\Omega_Q, T_Q)$ is defined for such~$Q$.
\begin{theorem}
\label{t:almost-fullness}
Let~$\Gamma$ be a countable additive subgroup of~$\RR$ containing~$1$, but not contained in~$\QQ$.
Then there is a map~$Q : \NN \cup \{ 0 \} \to \NN \cup \{ 0 \}$ which is increasing modulo intervals, such that~$Q(0) = Q(1) = 0$, such that for every integer~$k \ge 2$ we have~$Q(k) \le k - 2$ and such that the following property holds: the set~$\Omega_Q$ is a Cantor set and the unital ordered group~$G(\Omega_Q, T_Q)$ associated to~$(\Omega_Q, T_Q)$ is isomorphic to~$(\Gamma, \Gamma \cap \RR^+, 1)$.
\end{theorem}
To prove this theorem, let~$Q : \NN \cup \{ 0 \} \to \NN \cup \{ 0 \}$ be a map such that~$Q(0) = Q(1) = 0$, such that for each integer~$k \ge 2$ we have~$Q(k) \le k - 2$ and such that~$Q(k) \to \infty$ as $k \to \infty$.
We define an ordered Bratteli diagram $B_Q \= (V, E, \le)$, that was introduced by Bruin in~\cite[\S4]{Bru03}, as follows:
\begin{itemize}
\item
$V_0 \= \{ 0 \}$, $V_1 \= \{ k \in \NN : Q(k) = 0 \}$ and for $j \ge 2$,
$$ V_j \= \{ k \in \NN : k \ge j, Q(k-1) \le j - 2 \}; $$
\item
for an integer $j \in \NN$,
$$
E_j \= \{ j - 1 \to j \} \cup \{ j - 1 \to k : k \in V_{j} \setminus V_{j - 1} \}
\cup
\{ k \to k : k \in V_{j - 1} \setminus \{ j - 1 \} \}.
$$
\end{itemize}
Note that for every $j \in \NN$, we have~$j, j + 1 \in V_j$ and each vertex in~$V_j$ different from~$j$ has at most one edge arriving at it.
When~$j \ge 2$ the only edges arriving to $j \in V_j$ are $\{j - 1 \to j\}, \{j \to j \} \in E_j$.
So to define the partial order~$\ge$, we just have to define it, for each $j \ge 2$, between $\{j - 1 \to j \} \in E_{j}$ and $\{ j \to j \} \in E_{j}$: we put $\{ j - 1 \to j \} < \{ j \to j \}$.
The rest of the edges are maximal and minimal at the same time.

It is straight forward to check that the set~$X_{B_Q}$ is a Cantor set and that the infinite path $0 \to 1 \to 2 \to \cdots $ is the unique minimal path in~$B_Q$.
Therefore there is a well defined map $V_{B_Q} : X_{B_Q} \to X_{B_Q}$, see~\S\ref{ss:Bratteli diagrams}.
\begin{theorem}[\cite{Bru03}, Proposition~2]
\label{t:reduction to Bratteli}
Let~$Q : \NN \cup \{ 0 \} \to \NN \cup \{ 0 \}$ be a map such that~$Q(0) = Q(1) = 0$, such that for each integer~$k \ge 2$ we have~$Q(k) \le k - 2$ and such that~$Q(k) \to \infty$ as $k \to \infty$.
Consider the corresponding Bratteli-Vershik system $(X_{B_Q}, V_{B_Q})$ and generalized odometer $(\Omega_Q, T_Q)$ defined above.
Then there is a homeomorphism between~$X_{B_Q}$ and~$\Omega_Q$ that conjugates the action of~$V_{B_Q}$ on $X_{B_Q}$ to the action of~$T_Q$ on~$\Omega_Q$.
\end{theorem}
\begin{proof}[Proof of Theorem~\ref{t:almost-fullness}]
Let~$(q_n)_{n = 1}^\infty$ and~$(V_j)_{j = 1}^\infty$ be given by Proposition~\ref{p:infinitely generated} and put~$q_0 = 0$ and $V_0 = \{ 0 \}$.
By Part~2 of Proposition~\ref{p:infinitely generated} for each~$j \in \NN$ the least element of~$V_j$ is~$j$.
So for each~$k \in \NN$ there is at least one and at most finitely many integers $j \in \NN$ such that~$V_j$ contains~$k$.
Put~$Q(0) = 0$ and for~$k \in \NN$ let~$Q(k)$ be the smallest integer~$j \ge 0$ such that $k \in V_{j + 1}$.
Let us prove that the map~$Q : \NN \cup \{ 0 \} \to \NN \cup \{ 0 \}$ so defined is increasing modulo intervals.
By definition
$$ Q([q_1, q_2 - 1]) = Q([1, 2]) = \{ 0 \} = [q_0, q_1 - 1]. $$
Let~$n \ge 2$ be an integer and~$k \in [q_n, q_{n + 1} - 1]$.
By definition of~$j \= Q(k)$ we have~$k \in V_{j + 1} \setminus V_j$.
Part~1 of Proposition~\ref{p:infinitely generated} then implies that $V_{j + 1} \setminus V_j \subset [q_n, q_{n + 1} - 1]$ and that $Q(k) = j \in [q_{n - 1}, q_n - 1]$.
This completes the proof that~$Q$ is increasing modulo intervals.
We thus have~$Q(k) \to \infty$ as~$k \to \infty$.
Note that by Part~2 of Proposition~\ref{p:infinitely generated} for every~$j \in \NN$ we have~$V_{j} \setminus \{ j \} \subset V_{j + 1}$.
So by definition of~$Q$ for every~$j \in \NN$ we have~$Q^{-1}(j) = V_{j + 1} \setminus V_j$.
Since for each~$j \in \NN$ the minimal element of~$V_j$ is~$j$ and~$j + 1 \in V_j$, it follows that for every integer~$k \ge 2$ we have~$Q(k) \le k - 2$.
Furthermore, $(V_j)_{j = 0}^\infty$ coincides with the definition of the vertex set of the ordered Bratteli diagram~$B_Q$ defined above and for each~$j \in \NN$ the $j$-th transition matrix of~$B_Q$ is precisely $B(V_j, V_{j + 1})$.
Moreover, the~$0$-th transition matrix is equal to~$\bigl( \begin{smallmatrix} 1 & 1 \end{smallmatrix} \bigr)$.
In view of Theorem~\ref{t:reduction to Bratteli} and the remarks above, it follows that~$\Omega_Q$ is a Cantor set and that the last desired assertion is a direct consequence of Proposition~\ref{p:Bratteli} and Proposition~\ref{p:infinitely generated}.
\end{proof}
\subsection{From the generalized odometer to the post-critical set of a logistic map}
\label{ss:logistic}
Fix~$\parameter \in (0, 4]$ and let~$f_\parameter$ be the corresponding logistic map as defined in the introduction.
We assume that~$\parameter$ is sufficiently close to~$4$ so that~$f_\parameter^2(1/2) < 1/2 < f_\parameter(1/2)$.
Put~$c_0 = 1/2$ and for each integer $n \ge 1$ put $c_n = f_\parameter^n(c_0)$.
Define the sequence of compact intervals $(D_n)_{n = 1}^\infty$ inductively by $D_1 = [c_0, c_1]$ and for $n \ge 2$, by
$$ D_{n} =
\begin{cases}
f_\parameter(D_{n - 1}) & \text{if } c_0 \not \in D_{n - 1}; \\
[c_{n}, c_1] & \text{otherwise}.
\end{cases} $$
An integer $n \ge 1$ is called a \textit{cutting time} if $c_0 \in D_n$. 
We denote by $(S_k)_{k =0}^\infty$ the strictly increasing sequence of all cutting times.
From our assumption that $f_\parameter^2(c_0) < c_0 < f_\parameter(c_0)$ it follows that $S_0 = 1$ and $S_1 = 2$.

It can be shown that if~$S$ and~$S' \ge S + 1$ are consecutive cutting
times, then $S' - S$ is again a cutting time, which is less than or equal to~$S$ when~$f_\parameter$ has no periodic attractors, see for example~\cite{Bru95,Hof80}.
That is, if~$f_\parameter$ has no periodic attractors then for each $k \ge 1$ there is a non-negative integer~$Q(k)$ such that $Q(k) \le k - 1$ and $S_k - S_{k - 1} = S_{Q(k)}$. Putting $Q(0) = 0$, the function $Q : \NN \cup \{ 0 \} \to \NN \cup \{ 0 \}$ so defined is called the \textit{kneading map} of~$f_\parameter$.
Note that the sequence~$(S_k)_{k = 0}^\infty$ is defined from~$Q$ in the same way as in the definition of the generalized odometer~$\Omega_Q$.

Given $x = (x_k)_{k = 0}^{\infty} \in \Omega_Q$ and an integer $n \ge 0$,
put $\sigma(x|n) = \sum_{k = 0}^nx_kS_k$.
Observe that $\sigma(x|n)$ is non-decreasing with~$n$ and when~$x$ has infinitely many~$1$'s, $\sigma(x|n) \to \infty$ as $n \to \infty$.
On the other hand, if~$x$ has at most a finite number of~$1$'s, then $\sigma(x) \= \lim_{n \to + \infty} \sigma(x|n)$ is finite and $x = \expansion{\sigma(x)}$.

For $x = (x_k)_{k =0}^\infty$ different from $\expansion{0}$ we denote
by~$q(x) \ge 0$ the least integer such that $x_{q(x)} \neq 0$.
In~\cite[Theorem~1]{BruKelStP97} it is shown that for each $x \in \Omega_Q$ with infinitely many~$1$'s the sequence of intervals $(D_{\sigma(x|n)})_{n = q(x)}^\infty$ is nested and that $\bigcap_{n \ge q(x)} D_{\sigma(x|n)}$ is reduced to a point belonging to the \pcs~$X_{\parameter}$ of~$f_\parameter$.
Furthermore, if we denote this point by~$\pi(x)$ and for an integer $n \ge 0$ we put $\pi(\expansion{n}) = f_\parameter^n(c)$, then the map
$$ \pi : \Omega_Q \to X_{\parameter} $$
so defined is continuous and conjugates the action of~$T_Q$ on~$\Omega_Q$, to the action of~$f_\parameter$ on~$X_{\parameter}$.
\begin{proposition}
\label{p:projection-to-pcs}
Let $Q:\NN \cup \{ 0 \} \to \NN \cup \{ 0 \}$ be non-decreasing modulo intervals, such that~$Q(0) = Q(1) = 0$ and such that for every integer~$k \ge 2$ we have $Q(k) \leq k - 2$.
Then there is $\parameter \in (0, 4]$ such that~$Q$ is the kneading map of~$f_\parameter$ and the inverse of the factor map $\pi:\Omega_Q\to X_{\parameter}$ is defined on the complement of
$$ \mathcal{O}_{\parameter}
\=
\bigcup_{n = 0}^\infty f_{\parameter}^{-n}(1/2) $$
in~$X_{\parameter}$.
In particular~$\pi$ is injective on~$\Omega_Q \setminus \pi^{-1}(\mathcal{O}_{\parameter})$.
\end{proposition}
\begin{proof}
Since~$Q$ is non-increasing modulo intervals there is a strictly increasing sequence of integers~$(q_n)_{n = 0}^\infty$ such that~$q_0 = 0$, $q_1 = 1$ and such that for each~$n \in \NN$ we have~$Q([q_n, q_{n + 1} - 1]) \subset [q_{n - 1}, q_n - 1]$.

In Part~1 we prove that there is~$\parameter \in (0, 4]$ such that~$Q$ is the kneading map of~$f_\parameter$ and in Part~2 we prove that~$\pi^{-1}$ is well defined on~$X_{\parameter} \setminus \mathcal{O}_{\parameter}$.

\partn{1}
In view of~\cite{Bru95,Hof80}, to prove that there is~$\parameter \in (0, 4]$ such that~$Q$ is the kneading map of~$f_\parameter$ we just need to prove that for each $k \in \NN$ there is~$j_0 \in \NN$ such that
$$ Q(k + j_0) \ge Q(Q(Q(k)) + j_0) + 1 $$
and such that for each~$j \in \{1, \ldots, j_0 - 1 \}$
\begin{equation*}
\label{e:kneading equality}
Q(k + j) = Q(Q(Q(k)) + j).
\end{equation*}

Let~$k \in \NN$ be given.
Suppose first~$k \in [1, q_3 - 1]$, so that~$Q(k) \le q_2 - 1$ and thus~$Q(Q(k)) = 0$.
Then for each~$j \in [1, q_2 - 1]$ we have $Q(Q(Q(k)) + j) = Q(j) = 0$.
Since~$Q(k + j) \ge 0$ with strictly inequality when~$j = q_2 - 1$, this shows the desired assertion for this value of~$k$.
Suppose now $k \ge q_3$ and let~$n \ge 3$ be the integer such that $k \in [q_n, q_{n + 1} - 1]$.
Then $Q(k + 1) \ge q_{n-1}$.
On the other hand, $Q(k) \leq q_n - 1$ so $Q(Q(k)) + 1 \le q_{n-1} \le q_n - 1$ and
$$ Q(Q(Q(k)) + 1) \le q_{n - 1} - 1 \le Q(k + 1) - 1. $$
This proves the desired assertion for this value of~$k$, with~$j_0 = 1$.

\partn{2}
By~\cite[Lemma 11]{CorRiv10a}, to prove that~$\pi^{-1}$ is defined on~$X_{\parameter} \setminus \mathcal{O}_{\parameter}$, it is enough to show that for each sufficiently large integer~$k$ we have
\begin{equation}
\label{e:improved admissibility}
Q(k + 1) \ge Q(Q(Q(k)) + 1) + 2
\end{equation}
and that for each constant $K > 0$ and for every pair of distinct points $x, x'\in \Omega_Q$ that are not in the grand orbit of $\expansion{0}$, there is an integer $m$ satisfying
$$ \max\{q(T^m_Q(x)), q(T^m_Q (x'))\} \ge  K
\text{ and }
Q(q(T^m_Q(x)) + 1) \neq Q(q(T^m_Q(x')) + 1). $$
The proof of this last assertion follows the same arguments of those given in \cite[Lemma 17]{CorRiv10b}.
To finish the proof we will prove~\eqref{e:improved admissibility} for each integer~$k \ge q_3$.
Let~$n \ge 3$ be the integer such that $k \in [q_n, q_{n + 1} - 1]$.
As in Part~1 we have~$Q(k + 1) \ge q_{n - 1}$ and~$Q(Q(k)) \le q_{n - 1} - 1$.
On the other hand, $Q(k) \ge q_{n - 1}$, so~$Q(Q(k)) + 1 \ge q_{n - 2} + 1 \ge 2$ and 
$$ Q(Q(Q(k)) + 1) \le (Q(Q(k)) + 1) - 2 = Q(Q(k)) - 1 \le q_{n - 1} - 2 \le Q(k + 1) - 2. $$
This proves~\eqref{e:improved admissibility} for this value of~$k$ and completes the proof of the proposition.
\end{proof}

\begin{proof}[Proof of Theorem~\ref{t:fullness} and of Theorem~\ref{t:semi-equivalence}]
Let~$(X, T)$ be a uniquely ergodic minimal Cantor system.
Then the unital ordered group~$G(X, T)$ is a dimension group that has a unique state~$\phi$ mapping the order unit of~$G(X, T)$ to~$1$, see~\S\ref{ss:complete invariant}.
So by Lemma~\ref{l:unique state} the additive subgroup~$\Gamma \= \phi(G(X, T))$ of~$\RR$ is countable and~$\phi$ induces an isomorphism between the unital ordered group~$G(X, T) / \inf(G(X, T))$ and~$(\Gamma, \Gamma \cap \RR^+, 1)$.
Note that by definition~$\Gamma$ contains~$1$.
On the other hand~$\Gamma$ is acyclic because~$X$ is a Cantor set.

If~$\Gamma$ is not contained in~$\QQ$, then the desired assertions follow from Theorem~\ref{t:almost-fullness}, Proposition~\ref{p:projection-to-pcs} and~\cite[Theorem 2.2]{GioPutSka95}.

It remains to consider the case where~$\Gamma$ is contained in~$\QQ$.
Since~$\Gamma$ is acyclic, there exists a sequence of integers greater than or equal to two $(q_n)_{n = 1}^\infty$ such that the set
$$ \{1 / (q_1\cdots q_i) : i \in \NN \} $$
generates~$\Gamma$.
Define $Q:\NN\cup\{0\} \to \NN\cup\{0\}$ by $Q^{-1}(0) = [0,q_1-1]$ and for each~$i \in \NN$ by
$$ Q^{-1}(q_1 + \cdots + q_i)
=
[q_1+\cdots+q_{i}+1, q_1+\cdots+q_{i+1}]. $$
It is straight forward to check that the generalized odometer $(\Omega_Q,T_Q)$ corresponds to the odometer defined by the sequence $(q_n)_{n = 1}^\infty$, see for example~\cite[Lemma~8]{CorRiv10a}.
Therefore its unital ordered group is isomorphic to $(\Gamma, \Gamma\cap\RR^+,1)$, see for example~\cite{GjeJoh00}.
Since~$Q$ is non-decreasing, by~\cite{Bru95,Hof80} there is a parameter~$\parameter \in (0, 4]$ such that the kneading map of~$f_\lambda$ is equal to~$Q$.
It is well known that the \pcs{} of~$f_\lambda$ is conjugated to the odometer defined by the sequence $(q_n)_{n = 1}^\infty$.
So in this case the desired assertions follow from~\cite[Theorem~2.2]{GioPutSka95}.
\end{proof}

\bibliographystyle{alpha}

\end{document}